\documentclass[10pt, reqno, oneside, english]{smfart}

\usepackage{smfthm}

\usepackage[height=22cm, bottom=3.5cm]{geometry}
\usepackage{amsmath, amsthm, amssymb,amsfonts}
\usepackage[utf8]{inputenc}
\usepackage[english]{babel}
\usepackage[T1]{fontenc}
\usepackage{url}
\usepackage{braket}
\usepackage{enumerate}
\usepackage{xcolor}
\usepackage[colorlinks=true, linktoc=page, citecolor=blue, linkcolor=blue, urlcolor=blue]{hyperref}
\usepackage{lipsum}
\usepackage{graphicx}
\usepackage{nicematrix}
\usepackage{tikz}

\usepackage{setspace}
\renewcommand{\arraystretch}{1.15}

\numberwithin{equation}{section}

\newcommand{\C}{\mathbb{C}}
\newcommand{\Q}{\mathbb{Q}}
\newcommand{\Z}{\mathbb{Z}}
\newcommand{\N}{\mathbb{N}}
\newcommand{\OO}{\mathcal{O}}

\allowdisplaybreaks

\def\GL{\mathrm{GL}}
\def\SL{\mathrm{SL}}

\def\diag{\mathrm{diag}}
\def\wF{\widehat{F}^\times}
\def\tM{\widetilde{M}}
\def\tG{\widetilde{G}}
\def\dM{\mathcal{E}^2(M)}
\def\dtM{\mathcal{E}^2(\tM)}

\def\wF{\widehat{F}^\times}
\def\XMnru{\mathcal{X}_{nr}^u(M)}
\def\XtMnru{\mathcal{X}_{nr}^u(\tM)}
\def\S{\mathbb{S}}

\def \val{\mathrm{val}}

\def\1n{\{1, \dots , n\}}
\def\supp{\mathrm{supp}}
\def\chinru{\chi_{nr}^u}
\def\st{\mathrm{\;s.t\;}}

\def\Ind{\mathrm{Ind}}

\DeclareFontFamily{U}{mathx}{\hyphenchar\font45}
\DeclareFontShape{U}{mathx}{m}{n}{
      <5> <6> <7> <8> <9> <10>
      <10.95> <12> <14.4> <17.28> <20.74> <24.88>
      mathx10
      }{}
\DeclareSymbolFont{mathx}{U}{mathx}{m}{n}
\DeclareFontSubstitution{U}{mathx}{m}{n}
\DeclareMathAccent{\widecheck}{0}{mathx}{"71}

\title{On a basepoint issue in the tempered dual of~$p$-adic~$\SL(n)$}
\author{Marie Dautheville}
\address{Institut Élie Cartan de Lorraine, Nancy \& Metz, France}
\email{marie.dautheville@univ-lorraine.fr}

\begin{document}

\frontmatter
\begin{abstract} Let $F$ be a non-archimedean local field. Let $M$ be a Levi subgroup of $G = \SL(n,F)$. Consider the action on the discrete series of $M$ of the group of unitary unramified characters of $M$ (where a character acts by twisting). Given an orbit $ \OO$ for that action, let $W_\OO$ be the global stabilizer of $\OO$ in the Weyl group of $M$. We construct families of orbits $\OO$ for which the action of $W_\OO$ has no fixed point.

More precisely, let $n \geq 8$ be an integer that is not prime and not equal to $9$. We attach families of orbits without fixed points to any divisor $m \geq 2$ of $n$ such that $m \;| \; q-1$ where $q$ is the order of the residue field of $F$. The Levi subgroup $M \subset G$ for the corresponding families has $m$ blocks of size $k\geq 2$ and $2$ blocks of size $m$. The core of the paper is the construction of supercuspidal representations of $\GL(m,F)$ and $\GL(k,F)$ satisfying some conditions that we require to construct the families of orbits $\OO$ without fixed points.

In addition, we study the fixed point problem for the dimensions $n$ that we excluded above,~i.e.,~prime or small dimensions. We prove that there is always a fixed point under the action of~$W_\OO$ in those cases.

\end{abstract}


\maketitle

\title[On a basepoint issue in the tempered dual of~$p$-adic~$\SL(n)$]{On a basepoint issue in the tempered dual of~$p$-adic~$\SL(n)$}
\maketitle

\section{Introduction}

Let $F$ be a non-archimedean local field, and let $G$ be the $F$-points of a connected reductive group over $F$. Denote by $\widehat{G}_{temp}$ the tempered dual of $G$. Its connected components are indexed by $G$-conjugacy classes of pairs $(M, \sigma)$ where $M \subset G$ is a Levi subgroup and $\sigma$~is a discrete series representation of $M$. The component attached to a pair $(M, \sigma)$, denoted by~$\widehat{G}_{M,\sigma}$, consists of the irreducible reprensentations that occur in the parabolic induction to $G$ of a twist~$\sigma \otimes \chi$, where $\chi$ is a unitary unramified character of $M$.

In some special cases, the topology of the corresponding component $\widehat{G}_{M, \sigma}$ of $\widehat{G}_{temp}$ is well understood. This line of investigation was initiated by Plymen in the case of $\SL(2, \Q_p)$ in \cite{PSL2} and later for $\GL(n,F)$ in \cite{plymen1, plymen2} and for general reductive groups in \cite{plymen}. Other families of classical groups were considered by Plymen and his students: see e.g. \cite{PL,PJ}. More recently, new results have appeared in work of Afgoustidis and Aubert \cite{AAAMA}, Aubert and Plymen \cite{AP} and Clare and Crisp \cite{CC}. These results have connections with conjectures and theorems of Aubert, Baum, Plymen and Solleveld on the geometry of connected components of the smooth and tempered dual in various settings: see \cite{ABPS1, ABPS5, ABPS8, Solleveld} for surveys.

To be more specific about the connected components of the tempered dual attached to a pair~$(M, \sigma)$~as above, let $\dM$ be the set of equivalence classes of discrete series representations of $M$, let $\OO \subset \mathcal{E}^2(M)$ be the orbit $\{\sigma \otimes \chi, \; \chi \in \chinru(M)\}$ where $\chinru(M)$ denotes the group of unitary unramified characters of $M$. Denote by $W(M)$ the Weyl group of $M$ defined by~$W(M) =~N_G(M)/M$ where $N_G(M)$ is the $G$-normalizer of $M$. The Weyl group $W(M)$ of~$M$~acts on $\sigma$ by conjugaison, that is, for an element $w \in W(M)$, $$^w  \sigma = \sigma(w^{-1} \cdot w).$$ Let $W_\OO$ be the global stabilizer of $\OO$ in the Weyl group of $M$. Explicitly, \begin{equation*} W_\OO = \{ w \in W(M) \; | \; ^w\sigma \simeq \sigma \otimes \chi \mathrm{ \; for \; some \;} \chi \in \chinru(M)\}. \end{equation*} 

The geometry of $\widehat{G}_{M, \sigma}$ depends on the geometry of the action of $W_\OO$ on $\OO$. For instance, Afgoustidis and Aubert proved in \cite{AAAMA} that if this action has a fixed point with ``good properties'', then a simple structure theorem holds for $\widehat{G}_{M, \sigma}$. In this case such a fixed point plays the role of a ``natural base point'' for the connected component $\widehat{G}_{M, \sigma}$ of $\widehat{G}_{temp}$. This raises the following question: when does the action of $W_\OO$ on $\OO$ admit a fixed point? It is this question which we address in this paper, when $G$ is the group $\SL(n,F)$. When $G$ is a quasisplit classical group over $F$ (i.e. a quasisplit symplectic, orthogonal or unitary group), the authors of \cite{AAAMA} proved that the action of $W_\OO$ on $\OO$ always has a fixed point. But beyond this case, the question seems to have been little studied. In fact, to the author's knowledge, only one example of orbit without a fixed point is known: it appears in a paper of Roche \cite{roche}. In this example, $G = \SL(8,F)$, while $(M, \sigma)$ consists of a Levi subgroup isomorphic to $\left( \GL(2) \times \GL(2) \times \GL(4) \right) \cap \SL(8)$ and~$\sigma$~is a very particular supercuspidal representation of $M$. 

In this paper we show that Roche's example is not an isolated phenomenon. We construct infinite families of examples that do not have basepoints. Our examples occur in $\SL(n,F)$ for non-prime $n$ and are related to Roche's example (although our construction does not specialize to Roche's for appropriate $(M, \sigma)$). Conversely, for prime $n$, we show that there is always a fixed point.

\bigskip

To set the framework, we first recall some basic facts on discrete series of a Levi subgroup~$M~\subset~\SL(n,F)$~in Section \ref{1}. One essential point is Proposition \ref{pi}, originally from~\cite{GK}~and that we can find in a more general case in \cite{tadic} and in \cite{goldberg}, is that given a discrete series $\sigma$ of $M$ there exists a discrete series $\pi_\sigma$ of a Levi subgroup~$\tM \subset~\GL(n,F)$~such that $\sigma$ appears as a subrepresentation of the restriction of $\pi_\sigma$ to $M$. Then we can write $\pi_\sigma$ as a tensor product of discrete series of each block of $\tM$. The second important point is that the representation $\pi_\sigma$ is unique up to a torsion by a character of $\tM$ that is trivial on $M$, i.e. a character $\eta$ of $\wF$ composed with the determinant map on $\SL(n,F)$. We begin the construction of our families in Section \ref{2}. The core of the construction of our examples is to describe and control the following subgroup of $\wF$.

\bigskip

\begin{defi}\label{selftorsion} Let $m \geq 1$. Let $\rho \in \mathcal{E}^2(\GL(m,F))$. We define the self-torsion group associated to $\rho$, denoted by $X(\rho)$, as $$X(\rho) = \{ \mu \in \wF \; | \;  \rho \simeq \rho\otimes( \mu\circ \det)\}.$$ We shall write $\mu$ instead of $ \mu\circ \det$. The self-torsion group appears and plays an important role in \cite{ GK, goldberg, BKSLII}.
\end{defi}

\bigskip

In Section \ref{2}, we construct components without basepoints as follows. Let us fix an integer $n\geq 8$ and assume $n$ is not prime and $n \neq 9$. We partition $n$ as $mk +2m$ where $m$ is a nontrivial divisor of $n$ and $k \geq 2$.  Let $p$ be the characteristic of the residue field of $F$ and assume that $p \! \not | \; m$. Let $\varpi_F$ be an uniformizer in $F$. Assume there exists a character $\eta$ of $F^\times$ of order $m$ that is $\varpi_F$-totally ramified; i.e. $\eta(\varpi_F) =1$. To ease notation, we shall write totally ramified instead of $\varpi_F$-totally ramified. We consider the following Levi subgroup $\tM$ of $\GL(n,F)$: \begingroup\setlength{\arraycolsep}{1,5pt}\renewcommand{\arraystretch}{0.8} $$\tM = \begin{pmatrix}
    \GL(k) &&&& \\&\ddots  &&&\\&& \GL(k)&& \\ &&& \GL(m)& \\ &&&&  \GL(m)
\end{pmatrix}.$$ \endgroup Then we define a discrete series $\pi = \pi_1 \otimes \left( \pi_1 \otimes \eta \right) \dots \otimes \left( \pi_1 \otimes \eta^{m-1} \right) \otimes \pi_{m+1} \otimes \pi_{m+2}$ of $\tM$ where~$\pi_1 \in~\mathcal{E}^2(\GL(k,F))$ and $\pi_{m+1}, \pi_{m+2} \in \mathcal{E}^2(\GL(k,F))$ are such that \begin{equation}\label{prop}  X(\pi_1) = \{1 \}, \qquad X(\pi_{m+1}) = \left< \eta \right> \quad \mathrm{and}\qquad X(\pi_{m+2}) = \left< \chi\eta \right>  \end{equation} where $\chi$ is an unramified character of $F^\times$ of order $m$. Our main result is the following theorem.

\medskip

\begin{theo} The orbit $\OO$ attached to $\pi$ given above does not have a fixed point under the action of $W_\OO$. \end{theo}

\medskip 

The rest of Section \ref{2} is devoted to proving this theorem under assumptions we made above, i.e. there exists a totally ramified character $\eta$ of $F^\times$ of order $m$ and the representations $\pi_1, \pi_{m+1}$ and $\pi_{m+2}$ exist. This family of examples can be generalized to a broader family; see Remark \ref{braoder}.

Section \ref{3} is dedicated to constructing supercuspidal representations of the general linear group satisfying the assumptions in \eqref{prop}. The argument is, in essence, sketched in \cite{roche} and uses the Bushnell and Kutzko theorem of types. Given a totally ramified field extension $E/F$ and an additive character $\psi$ of $F$, using \cite{BK}, we construct a simple stratum, then a simple type $(J, \lambda)$. Then by a compact induction of $\lambda$ to $\GL(m,F)$, we obtain a supercuspidal representation of $GL(m,F)$ and we can make its self-torsion group explicit. This subgroup of $\wF$ is controlled by the norm map $N_{E/F} : E^\times \longrightarrow F^\times$ of $E^\times$. The last step is to fix explicit extensions of $F$ to construct $\pi_{m+1}$ and $\pi_{m+2}$ and control their self-torsion groups. 

Finally, Section \ref{4} treats the cases that we excluded in Section \ref{2}, that is to say, all prime numbers and small dimensions that are not prime $(4, 6, 9)$. The mains results of this section prove that for any Levi subgroup $M$ of $\SL(n,F)$ with $n$ prime or equal to $4, 6$ or $9$, and any orbit $\OO \subset \dM$, there is always a fixed point in $\OO$ under the action of $W_\OO$. For prime dimension, the argument is combinatorial in nature: analyzing the arithmetic of the size of blocks in a Levi subgroup of $\SL(n)$, we show that the fixed point problem simplifies and that the method of \cite{AAAMA} can be used to obtain a fixed point. The remaining small dimensions are treated by a case-by-case analysis: starting with a discrete series $\sigma$ of a Levi subgroup, we construct a fixed point in the orbit $\OO$ of $\sigma$.

\medskip

\subsection*{Acknowledgements} The author would like to sincerely thank her PhD supervisors, Alexandre Afgoustidis and Angela Pasquale, for their careful reading of the manuscript and for the valuable discussions and exchanges throughout this work. The author also wishes to thank Anne-Marie Aubert and Alan Roche for their careful reading of this paper,their helpful comments and suggestions, and their contribution to its revision. Finally, the author thanks Roger Plymen for his time and stimulating discussions, which have helped enrich this work.

This work was supported by the project OpART (ANR-23-CE40-0016) of the Agence Nationale de la Recherche.

This manuscript was written without the use of any AI except for spelling and language correction.

\section{Discrete series of a Levi subgroup of $\SL(n,F)$ and their Weyl group}\label{1}

\subsection{Levi subgroups of the special linear group} 

Fix an integer $n\geq 1$. We denote by $G$ the group $\SL(n,F)$ and $\widetilde{G}$ the group $\GL(n,F)$. We will give a quick description of the Levi subgroups of $G$ and $\widetilde{G}$. We denote by $\tM$ a Levi of $\widetilde{G}$. There exists a partition $n = n_1 +n_2 + \dots + n_r$ such that $\tM$ is conjugate to \begingroup\setlength{\arraycolsep}{1,5pt}\renewcommand{\arraystretch}{1} $$ \left[\begin{matrix}
    \GL(n_1, F) &&&\\ &\GL(n_2, F)&& \\ &&\ddots& \\ &&&\GL(n_r, F) 
\end{matrix}\right].$$ In particular, $\tM$ is isomorphic to $\mathrm{GL}(n_1, F) \times \dots \times \mathrm{GL}(n_r, F)$. \endgroup Every Levi subgroup of $G$ has the form $\tM \cap G$ where $\tM$ is a Levi subgroup of $\tG$; so up to conjugation, \begingroup\setlength{\arraycolsep}{1,5pt}\renewcommand{\arraystretch}{1}  \begin{equation*}  M = \left[\begin{matrix}
   \GL(n_1, F)&&&\\ &\GL(n_2, F)&& \\ &&\ddots& \\ &&&\GL(n_r, F)
\end{matrix}\right]_{\det = 1}. \end{equation*} \endgroup

The Weyl groups of $M$ and $\tM$, denoted by $W(M)$ (resp. $W(\tM)$), are isomorphic and acts by permutations on the diagonal blocks of $M$ (resp. $\tM$). In this way, the Weyl group of $M$ (resp.~$\tM$) appears as a subgroup of the symmetric group $\mathfrak{S}_r$ and is generated by transpositions~$(i \; j)$ such that the blocks $n_i$ and $n_j$ have the same size.

\bigskip

\subsection{Discrete series of $M$ as a subrepresentation of a Levi subgroup of $\tM$}

We now recall the results from \cite{tadic} that will be needed in what follows. Those results are originally from \cite{GK} for the characteristic zero but Tadìc generalized them to positive characteristic.

\bigskip

\begin{prop}[\cite{tadic}, Proposition 2.2]\label{pi} Let $\sigma$ be an irreducible smooth representation of $M$, there exists  an irreducible smooth representation $\pi_\sigma$ of $\tM$ such that $\sigma$ is isomorphic to a subrepresentation of $\pi_{\sigma|M}$. If $\sigma$ is a discrete series of $M$, then there exists a discrete series $\pi_\sigma$ of $\tM$ such that $\sigma$ is isomorphic to a subrepresentation of $\pi_{\sigma|M}$.
\end{prop}

\bigskip

 The following proposition proves that, given $\sigma$, the representation $\pi_\sigma$ in Proposition \ref{pi} is, unique up to a twisting by a character of $F^\times$. 
Let $\OO_{\widetilde{M}}(\pi_1)$ denote the set of $\tau \in \mathcal{E}^2(M)$ isomorphic to a subrepresentation of $\pi_{1|M}$.

\bigskip

\begin{prop}[\cite{tadic}, Corollary 2.5]\label{tadic} Let $\pi_1,\pi_2 \in \mathcal{E}^2(\widetilde{M})$. The following assertions are equivalent: 
\begin{enumerate}
\item There exists a character $\eta$ of $\widetilde{M}$ trivial on $M$ such that $\pi_1 \simeq \pi_2 \otimes \eta$. 
\item $\OO_{\widetilde{M}} (\pi_1) \cap \OO_{\widetilde{M}}(\pi_2) \neq \emptyset$.
\item $\OO_{\widetilde{M}}(\pi_1) = \OO_{\widetilde{M}}(\pi_2)$.
\end{enumerate} 
\end{prop}

\bigskip

\begin{rema} The characters of $\widetilde{M}$ trivial on $M$ are exactly the characters of the form $\eta \circ \det$ with $\eta \in \wF$. In the following, if no confusion is possible, we will write $\eta$ for $\eta \circ \det$. \end{rema}

\bigskip

\subsection{Structure of a discrete series of $\tM$ restricted to $M$}

Let $\pi$ be a discrete series of a Levi subgroup $\tM$ of $\tG$. Define $X(\pi)$ as in Definition \ref{selftorsion} with $\tM$ in place of $\GL(m,F)$. We state without proof some results on the discrete series of $\tM$ restricted to $M$ and, in particular, a result on the length of $\pi_{|M}$ that depends on the group $X(\pi)$.

\bigskip

\begin{prop}[\cite{tadic}, Lemma 2.1] Let $\pi\in \dtM$. The restriction of $\pi$ to $M$ is a finite direct sum of irreducible representations of $M$. Moreover, the multiplicity of each subrepresentation is the same. We have $$\pi_{|M} \simeq m_\pi \! \bigoplus_{\tau \in \OO_{\widetilde{M}}(\pi)} \tau$$ where $m_\pi$ is that ``global'' mutiplicity. \end{prop}

\bigskip

\begin{prop}[\cite{BKSLII}, Corollary 1.6 \& Remark (ii)]\label{cyclic}
 If $X(\pi)$ is cyclic, then $\pi_{|M}$ is multiplicity free; that is to say, $m_\pi =1$.\end{prop}

\bigskip

\begin{prop}[[\cite{BKSLII}, Corollary 1.6 \& Remark (ii)]\label{length} Let $\pi \in \dtM$ and assume $\pi_{|M}$ is multiplicity free. Then, the length of $\pi_{|M}$ is equal to the index in $\tM$ of the group $\mathcal{T}(\pi)$, defined by $\mathcal{T}(\pi) =~\bigcap_{\mu \in X(\pi)} \ker ( \mu \circ \det)$.  
\end{prop}

\bigskip

\subsection{The group $W_\OO$ and its action on $\pi_\sigma$}

\subsubsection{Characters of $M$ and $\tM$.}

\bigskip

\begin{prop}\label{characterextended} Let $\chi \in \XMnru$. There exists an extension $\widetilde{\chi} \in \XtMnru$ such that~$\widetilde{\chi}_{|M} =~\chi$. 
 \end{prop}
 
\bigskip

\begin{proof} The author thanks Alan Roche for suggesting this more elegant alternative proof.
We want to prove that the restriction homomorphism $ \chinru(\tM) \longrightarrow \chinru(M)$ is surjective. Consider the surjective homomorphism $$\begin{matrix} M & \longrightarrow & \Z^d \\ \diag(g_1, \dots, g_r) & \longmapsto & (\val(\det(g_1)), \dots, \val(\det(g_r)) \end{matrix}$$ and denote by $\tM_1$ its kernel. Then $$\chinru(\tM) = \mathrm{Hom}(\tM/ \tM_1, \S^1) \cong  \mathrm{Hom}(\Z^r, \S^1).$$ By a similar argument, we have $$\chinru(M) =  \mathrm{Hom}(M/M_1, \S^1)$$ where $M_1 = \tM_1 \cap M$. Since $\det(g_1) \times \dots \times \det(g_r) = 1$, the image of the restriction to $M$ of the previous homomorphism is $$\Z_0^r = \{ (z_1, \dots, z_r) \in \Z^r \; | \; z_1 + \dots + z_r = 0 \}.$$ Therefore, $$\chinru(M) \cong  \mathrm{Hom}(Z_0, \S^1).$$ Note that $\Z^r \cong \Z_0^r \oplus \Z$ and any homomorphism from $\Z_0^r$ into any group $\mathcal{G}$ extends to one from $\Z^r$ to $\mathcal{G}$. This proves the proposition.
\end{proof}

\bigskip

 \begin{coro}\label{character}
Let $\chi$ be a unitary unramified character of $M$, then $\chi$ can be expressed as  \begin{equation}\chi \; : \;M \ni \diag (g_1, \dots , g_r) \longmapsto \prod_{k = 1}^r \; z_k^{\mathrm{val}(\mathrm{det} g_i)}\end{equation} with $z_1, \dots z_r \in \mathbb{S}^1$.
\end{coro}

\bigskip 

\begin{proof} 
Any unitary unramified character $\lambda$ of $\tM$ can be written, in an unique way, as \begin{equation}\lambda  \; : \; \tM \ni \diag (g_1, \dots , g_r) \longmapsto \prod_{k = 1}^r \; z_k^{\mathrm{val}(\mathrm{det} g_i)}\end{equation} with $z_1, \dots, z_r \in \S^1$. By the previous the previous proposition, any unit	ary unramified character of $M$ can be written as a restriction of a character of $\chinru(\tM)$. This proves the corollary.
\end{proof}

 In fact, since for all $z \in \S^1$ and for any matrix $\diag(g_1, \dots , g_r)$ in $M$, we have $$\prod_{k = 1}^r \; (zz_k)^{\mathrm{val}(\mathrm{det} g_i)} = \prod_{k = 1}^r \; z_k^{\mathrm{val}(\mathrm{det} g_i)},$$ the morphism $$\begin{array}{cll}   \XtMnru & \longrightarrow & \XMnru \\   (z_1, \dots, z_r) & \longmapsto & \chi(\cdot)= \prod_{i=1}^r z_i^{\mathrm{val}(\det(\cdot))}\end{array}$$ is surjective and its kernel is $\{(z, z, \dots , z), \; z \in \S^1\}$. Hence, we obtain $$\chinru(M) \cong  (\mathbb{S}^1)^r / \mathbb{S}^1.$$ 
 
 \bigskip

\subsubsection{Description of $W_\OO$}

Let $\sigma \in \dM$ and choose $\pi_\sigma \in \dtM$ such that $\sigma \subset \pi_{\sigma|_M}$. This way we can decompose $\pi_\sigma$ as $$\pi_\sigma \simeq \pi_1 \otimes \dots \otimes \pi_r$$ where $\pi_i \in \mathcal{E}^2(\GL(n_i,F))$. We now describe $W_\OO$ and $W_\sigma$ more explicitly. For $W_\sigma$, this is a result of Goldberg \cite{goldberg}.
 
\bigskip

\begin{prop}[\cite{goldberg}, Lemma 2.3] Let $\sigma \in \dM$ and let $\pi_\sigma \in \dtM$ be such that $\sigma \subset \pi_{\sigma |M}$. Then  $$W_\sigma = \{w \in W(M), \exists \eta \in \wF  \st ^w\pi_\sigma \simeq \pi_\sigma \otimes (\eta \circ \det) \}.$$
\end{prop}

\bigskip

For $W_\OO$, we can give a similar description involving a twist by an unramified character of $\tM$. The proof can be find in [\cite{GR}, Lemma 2.1] for supercuspidal representations but the same argument works for the general case.

\bigskip

\begin{prop}\label{WTheta}
Let $\sigma \in \dM$ and let $\pi_\sigma \in \dtM$ be such that $\sigma \subset \pi_{\sigma |M}$. Then $$W_\OO = \{w \in W(M), \exists \eta \in \wF, \exists \widetilde{\chi} \in \XtMnru \st ^w\pi_\sigma \simeq \pi_\sigma \otimes \widetilde{\chi} (\eta \circ \det) \}.$$
\end{prop}

\bigskip




\begin{rema} Let $w \in W_\OO$. Then by Proposition \ref{WTheta}, there exists $\chi \in \chinru(\tM)$ and $\eta \in \wF$ such that $^w\pi_\sigma \simeq \pi_\sigma \otimes \widetilde{\chi} (\eta \circ \det)$. We give some remarks on $\eta$. Passing to the central character, we have that $\chi^n\eta^n =1$, since $\chi$ is unitary, $\eta$ is also unitary. With the properties seen in Appendix \ref{appendix}, we can decompose $\eta$ as the product of its ramified and unramified parts, denoted by $\eta_{ram}$ and $\eta_{nr}$, and we can choose $\eta$ totally ramified, i.e. $\eta(\varpi_F) =1$, by replacing $\chi$ by $\chi \eta_{nr}$ and $\eta$ by $\eta_{ram}$. Note that the order of $\eta_{ram}$ divides $n$. \end{rema}

\bigskip

\section{Construction of orbits without a fixed point}\label{2}

Let $n \geq 8$ be a non-prime integer distinct from $9$. (We will discuss small dimensions and prime values of $n$ in Section \ref{4}.) Let $m$ be a divisor of $n$ such that $n/m \geq 4$. Let $F$ be a finite field extension of $\Q_p$ and $G=\SL(n, F)$. We consider the following Levi subgroup~$\tM$ of $\tG =~\GL(n, F)$: \begingroup\setlength{\arraycolsep}{1,5pt}\renewcommand{\arraystretch}{0.9}  $$ \tM = \begin{pmatrix}
    \GL(k,F) &&&& \\&\ddots  &&&\\&& \GL(k,F)&& \\ &&& \GL(m,F)& \\ &&&&  \GL(m,F)
\end{pmatrix}$$ \endgroup with $k= \frac{n}{m}-2$.

Let $\sigma \in \dM$ with $M = \tM \cap \SL(n,F)$ and let $\OO$ be the orbit of $\sigma$ under the action of $\XMnru$. We recall that saying $\OO$ has a fixed point under the action of $W_\OO$ means that there exists $\lambda \in \XMnru$ such that for all $w \in W_\OO$,  $^w(\sigma \otimes \lambda) \simeq \sigma \otimes \lambda$.  In other words, there exists $\lambda \in \XtMnru$ such that $W_\OO = W_{\sigma \otimes \lambda}$. 

The purpose of this section is to construct a discrete series $\pi$ of $\tM$ whose restriction $\sigma$ to $M$ is a discrete series of $M$ and whose orbit $\OO$ has no fixed point. We recall that for $\tau \in \mathcal{E}^2(\GL(m,F))$, $m \in \N^*$, the self-torsion group of $\tau$ (Definition \ref{selftorsion}) is $$X(\tau)=\{ \mu \in \wF \st \tau \simeq \tau\mu\}.$$ 

The core of the construction relies on the precise structure of those $X(\tau)$ groups. The following proposition provides examples of representations $\tau$ with prescribed $X(\tau)$. We will use them to construct orbits which do not have any fixed point. 
 
 \bigskip

\begin{prop}\label{main} Assume there exists a totally ramified character $\eta \in \wF$ of order~$m$. Let~$\chi \in \mathcal{X}_{nr}^u(\wF)$ be an unramified character of $\wF$ of order $m$. 

 \noindent Then there exist~$\pi_1 \in \mathcal{E}^2(\GL(k,F))$ and~$\pi_{m+1},  \pi_{m+2} \in~\mathcal{E}^2(\GL(m,F))$ such that \begin{equation*} X(\pi_1)= \{ 1\}, \qquad X(\pi_{m+1})=\left< \eta\right> \qquad \mathrm{and} \qquad X(\pi_{m+2}) = \left< \chi\eta\right>.\end{equation*}
 \end{prop}

\bigskip

We postpone the proof of this proposition in Section \ref{3}. To the remainder of this section, we apply Proposition \ref{main} to construct the aforementioned examples of orbits without fixed points. Let $\eta \in \wF$ be of order $m$, $\chi \in \mathcal{X}_{nr}^u(M)$ be of order $m$ and choose $\pi_1 \in \mathcal{E}^2(\GL(k,F))$, $  \pi_{m+1},  \pi_{m+2} \in \mathcal{E}^2(\GL(m,F))$ satisfying the properties in Proposition \ref{main}.

For all $2 \leq i \leq m$, set $\pi_i = \pi_1 \eta^{i-1}$ and set $\pi = \pi_1 \otimes \pi_2 \otimes \dots \otimes \pi_{m+1} \otimes \pi_{m+2}$; thus, \begingroup\setlength{\arraycolsep}{1,5pt}\renewcommand{\arraystretch}{0.8} \begin{equation}\label{pii}\pi =   \begin{pmatrix}
    \pi_1 &&&& \\ & \pi_1\eta &&& \\ && \ddots&&&\\  &&&\pi_{1}\eta^{m-1}& &\\ &&&& \pi_{m+1}& \\ &&&&& \pi_{m+2}
\end{pmatrix} \in \dtM.\end{equation} Denote by $\widetilde{\OO}$ the orbit of $\pi$ under the action of $\XtMnru$ and let $\OO$ be the $\XMnru$-orbit of any discrete series $\sigma \in \dM$ such that $\sigma \subset \pi_{|M}$. \endgroup

\bigskip

\begin{rema} We will discuss the existence of characters of $F^\times$ of order $m$ in Appendix \ref{appendix}. \end{rema}

\bigskip

\begin{prop}\label{WT} We have $W_\OO = \left<  (1 \; \dots \; m) \right>$.
\end{prop}

\bigskip

In proving the proposition, we use the following lemma, which follows directly from the definitions:\\

\begin{lemm}\label{selftorsiongrp} Let $\tau_1, \tau_2 \in \mathcal{E}^2(\GL(n,F))$ be such that $\tau_1 \simeq \tau_2\alpha$ where $\alpha \in \wF$. Then~$X(\tau_1)~=~X(\tau_2)$.
\end{lemm}

\bigskip

\begin{proof}[Proof of Proposition \ref{WT}]
First, let us check that the $m$-cycle $ (1 \; \dots \; m)$ is in  $W_\OO$. Since $\pi_{m+1} \simeq \pi_{m+1} \eta$ and $\pi_{m+2} \simeq \pi_{m+2} \chi\eta$, we have

\begingroup\setlength{\arraycolsep}{1,5pt}\renewcommand{\arraystretch}{0.8} \begin{equation}\label{12} ^{(1 \; \dots \; m)} \pi  \simeq  {\footnotesize \begin{pmatrix} \pi_1\eta &&&&&& \\ & \pi_1\eta^2 && \\ && \ddots & \\ &&& \pi_1\eta^{m-1} \\ &&&& \pi_1 \\ &&&&& \pi_{m+1} & \\ &&&&&& \pi_{m+2}
\end{pmatrix} \simeq \pi \otimes  \begin{pmatrix} 1 &&& \\ & \ddots && \\ && 1 & \\ &&& \chi
\end{pmatrix} }\otimes \eta \end{equation}\endgroup and $ (1 \; \dots \; m) \in W_\OO$, as claimed. 

Now, let $w \in W_\OO$ be arbitrary. Then, there exists~$\lambda~=~\otimes_{i =1}^{m+2} \lambda_i \in \mathcal{X}_{nr}^u(\tM)$ and $\mu \in \wF$ totally ramified such that for all $1 \leq i \leq m+2$, $$\pi_{w(i)} = \pi_i \otimes \lambda_i \mu.$$ By Lemma~\ref{selftorsiongrp}, $X(\pi_{w(i)}) = X(\pi_i)$. But we know that the self-torsion groups of $\pi_1, \pi_{m+1}$ and~$\pi_{m+2}$ are different therefore there does not exist any character $\alpha$ of $F^\times$ such that $\pi_i \simeq \pi_j \otimes \alpha$ for~$i \neq~j \in~\{1,m+1, m+2 \}$ . Hence, $w$ must fixe $m+1$ and $m+2$. Now write $w(1) = 1 +a$ for $0 \leq a < m$. Then, by assumption, we have $$\pi_{w(1)} \simeq \pi_1 \lambda_1 \mu \qquad \mathrm{and} \qquad \pi_{w(j)} \simeq \pi_j \lambda_j \mu$$ and by \eqref{12}, $$ \pi_{w(1)} \simeq \pi_1 \eta^{w(1)-1}  \qquad \mathrm{and} \qquad \pi_{w(j)} \simeq \pi_1 \eta^{w(j)-1}.$$ 
Hence $$\pi_1 \simeq \pi_1  \lambda_1^{-1} \mu^{-1} \eta^{w(1)-1} \qquad \mathrm{and} \qquad \pi_j  \simeq \pi_1 \lambda_j^{-1} \mu^{-1} \eta^{w(j)-1}.$$ Since $\pi_j \simeq \pi_1 \eta^{j-1}$, this yields $$\pi_1  \lambda_1^{-1} \mu^{-1} \eta^{w(1)-1} \simeq \pi_1 \lambda_j^{-1} \mu^{-1} \eta^{w(j)-j}.$$ Hence, $$\pi_1 \simeq \pi_1  \lambda_1\lambda_j^{-1} \eta^{w(j)-w(1)-j+1}.$$ Since $X(\pi_1)=\{1\}$ and $\eta$ is totally ramified, this implies $$\lambda_1 = \lambda_j \qquad \mathrm{and} \qquad  \eta^{w(j)-(a+j)} = 1.$$ Since $\eta$ is totally ramified of order $m$, we have $w(j) = j + a \! \mod m$ and $w = ((1 \; 2 \; \dots \; m)^a$. Hence, $W_\OO = \left<  (1 \; \dots \; m) \right>$.
\end{proof}

\bigskip

\begin{theo}\label{thm} The orbit $\widetilde{\OO}$ of $\pi$ does not have any fixed point under the action of $W_\OO$.
\end{theo}

\bigskip

\begin{proof} Assume that there exists $\lambda \in \XtMnru$ such that $\pi \otimes \lambda$ is fixed under the action of $W_\OO$. Then there exists $\mu \in \wF$ such that $^{ (1 \; \dots \; m)}(\pi \otimes \lambda) \simeq \pi \otimes \lambda\mu$. So \begin{equation}\label{fixedpoint} \left\{\begin{matrix}\pi_2 \lambda_2 \simeq \pi_1\lambda_2 \eta \simeq \pi_1 \lambda_1 \mu \\ \pi_{m+1} \lambda_{m+1} \simeq \pi_{m+1} \lambda_{m+1} \mu \\ \pi_{m+2} \lambda_{m+2} \simeq \pi_{m+2} \lambda_{m+2} \mu  \end{matrix}\right.\end{equation} this implies that $\mu \in X(\pi_{m+1}) \cap X(\pi_{m+2})$. Therefore, necessarily, $\mu= 1$. 

Using the first equation of \eqref{fixedpoint} with $\mu=1$,  we have $\pi_1\eta \simeq \pi_1 \lambda_1 \lambda_2^{-1}$. Since $X(\pi_1)=\{1\}$, we obtain $\eta =  \lambda_1 \lambda_2^{-1}$, therefore, $\eta$ is unramified, which is a contradiction because $\eta$ is totally ramified and nontrivial.
\end{proof}

\bigskip

\begin{coro} The restriction of $\pi$ to $M$ is an irreducible discrete series of $M$ and has no fixed point under the action of $W_\OO$. \end{coro}

\bigskip

\begin{proof} Since $X(\pi)$ is trivial, by Proposition \ref{cyclic}, $\pi_{|M}$ is multiplicity free and by Proposition~\ref{length}, since $\mathcal{T}(\pi) = \ker (\mathrm{Triv}) = M$, the length of $\pi_{|M}$ is one. Finally, by the previous theorem, $\pi_{|M}$ has no fixed point under the action of $W_\OO$. \end{proof}

\begin{rema}\label{braoder}
The families of representations that we set above can be even more generalized: let $b$ be a positive integer and consider the Levi subgroup $$\tM \cong \GL(k,F)^b \times \dots \GL(k,F)^b \times \GL(m,F) \times \GL(m,F) \subset \GL(n,F)$$ with $n = bkm +2m$. Let $\pi_1, \pi_{m+1}$ and $\pi_{m+2}$ be as in Proposition \ref{main} and let $$\pi = \pi_1^{\otimes b} \otimes (\pi_1\eta)^{\otimes b} \otimes \dots \otimes (\pi_1\eta^{m-1})^{\otimes b} \otimes \pi_{m+1} \otimes \pi_{m+2}$$ where $\tau^{\otimes b}$ is $\tau \otimes \dots \otimes \tau$ with $b$ factors. Let $\OO$ be the orbit of $\pi$ under the action of $\chinru(\tM)$. By following the line of argument of this section, we have by a similar argument used in the proof of Proposition \ref{WT}, that $$W_\OO \cong \underbrace{\mathfrak{S}_b \times \dots \times \mathfrak{S}_b}_{m \mathrm{ \; times}} \rtimes \left< (1 \; 2 \dots m) \right>.$$ The same results of this section apply to $\pi$ and $\OO$, specially Theorem \ref{thm}, which says that $\OO$ does not have a fixed point under the action of $W_\OO$.
\end{rema}

\section{Construction of certain supercuspidal representations of $\GL(m,F)$.}\label{3}

Let $F$ be a non-archimedean local field. Denote by $ \OO_F$ its ring of integers and by $\mathfrak{p}_F$ the unique maximal ideal of $\OO_F$. Let $q = p^f$ be the cardinal of the residue class field $k_F = \OO_F/ \mathfrak{p}_F$. Assume that $p \not \! | \;  m$ and $m \; | \; q-1$. We want to construct two supercuspidal representations $\pi_{m+1}$ and $\pi_{m+2}$ of $\GL(m,F)$ which satisfy the assumption in Proposition \ref{main}; in particular, we want $X(\pi_{m+1})$ and $X(\pi_{m+2})$ to be isomorphic to $\Z/m\Z$. To construct them, we will proceed in several steps, following the book of Bushnell and Kutzko~\cite{BK}. This construction is sketched in \cite{roche}, but we shall try to make the details fully explicit here.

We begin with a finite field extension of degree $m$, of the shape $E = F[\varpi]$ such that the minimal polynomial of $\varpi$ satisfies Eisenstein's criterion. With this set up, we construct a simple stratum. Then, we will define a compact open subgroup $J$ of $\GL(m,F)$ and construct a character $\lambda$ of $J$, in this way, the pair $(J, \lambda)$ will be a simple type in $\GL(m,F)$. Finally, by compact induction, we will obtain the desired supercuspidal representations.

\subsection{Construction of a simple straum}

Let $\mathcal{M}_m(F)$ be the set of matrices of size $m$ with coefficients in $F$. The first step is to construct a simple stratum from an extension $E/F$ of degree $m$ and a minimal element $\beta \in \mathcal{M}_m(F)$ over $F$. Recall that $F$ is a non-archimedean local field. Denote by $\OO_F$ its ring of integers, $\mathfrak{p}_F$ is the unique maximal ideal of $\OO_F$ and $k_F= \OO_F/\mathfrak{p}_F$ is the residue class field of $F$. Denote by $\varpi_F$ a uniformizer in $F$, that is to say, $v_F(\varpi_F) = 1$ with $v_F$ the valuation on $F$. For every element $x$ of $F^\times$, we can write $x = \varpi_F^n u$ with $n \in \Z$ and $u \in \OO_F^\times$ with $\OO_F^\times$ the set units in $F^\times$. The structure of $F^\times$ will be discussed in Appendix \ref{appendix}. Let $V$ be an $F$-vector space of dimension $m$ and set $A= \mathrm{End}_F(V)$ and $G= \mathrm{Aut}_F(V)$, which we respectively identify with $\mathcal{M}_m(F)$ and $\GL(m,F)$.

\bigskip

\begin{defi}[\cite{BK}, p.20]\label{definition}  An $\OO_F$-lattice chain in $V$ is a sequence $\mathcal{L} = \{ L_i, \; i \in~\Z\}$ of $\OO_F$-lattices in $V$ such that \begin{enumerate} \item for all $i \in \Z$, $L_i  \supset  L_{i+1}$ and $L_i  \neq  L_{i+1}$; 
\item there exists $e \in \Z$ such that for all $i \in \Z$, $L_{i+e} = \varpi_FL_i$. \end{enumerate}
 Let $\mathcal{L}$ be an $\OO_F$-lattice chain order in $V$. The corresponding $\OO_F$-(hereditary) order in $\mathcal{M}_m(F)$ is the set $$\mathfrak{A} = \mathfrak{A}(\mathcal{L}) = \{ x \in \mathcal{M}_m(F) \; | \; x L_i \subset L_i, \; \forall i \in \Z\}.$$
\end{defi}

\bigskip

Let $\mathfrak{A}= \mathfrak{A}(\mathcal{L})$ be an $\OO_F$-order as above; we define its Jacobson radical, denoted by $\mathfrak{P}$, by $$\mathfrak{P} = \{ x \in A \; | \; x L_i \subset L_{i+1}\}$$ and set $$\mathfrak{P}^n =  \{ x \in A \; | \; x L_i \subset L_{i+n}\}, \qquad n \in \Z.$$

Set $$\mathfrak{K}(\mathfrak{A}) = \{ x \in \GL(m,F) \; | \; x^{-1} \mathfrak{A} x = \mathfrak{A}\}.$$ This is a compact-mod-center subgroup of $G$. And for $x \in A$ define, $$v_\mathfrak{A}(x) = \max \{ n \in \Z \; | \; x \in \mathfrak{P}^n\},$$ the valuation map associated with the order $\mathfrak{A}$. To any order, we may also associate an integer $k_0(\beta, \mathfrak{A})$ and define a notion of minimal element in $A$. To prevent heavy notations and definitions, we shall not spell out the corresponding definitions and refer the reader, respectively, to \cite{BK}, Definition 1.4.5 and Definition 1.4.14 for more details.
\bigskip

\begin{rema}\label{embedding} Suppose $E$ is a totally ramified finite field extension of $F$ of degree $m$. For a given uniformizer $\varpi_E \in E$, we can choose as a basis in $E$ the set $\{1, \varpi_E, \varpi_E^2, \dots, \varpi_E^{m-1}\}$. This way, the field extension $E$ can be embedded in $A$ by the following map: $$\begin{matrix}  E & \longrightarrow & A \\ a & \longmapsto & \left[\begin{matrix} \mu_a :& F & \longrightarrow & F \\ & x & \longmapsto & ax \end{matrix}\right] \end{matrix}.$$ We shall identify the field $E$ with the image of this map, thereby viewing $E$ as a subset of $\mathrm{End}_V(F)$ or even $\mathcal{M}_m(F)$.
 \end{rema}
 
 \bigskip 
 
\begin{defi}[\cite{BK}, Definition 1.5.5] A stratum is a $4$-tuple $[\mathfrak{A}, n, r, \beta]$ where $\mathfrak{A}$ is an $\OO_F$-order in $A$, $n,r$ are integers such that $n>r$ and $\beta$ is an element of $A$. We say that $[\mathfrak{A}, n, r, \beta]$ is a simple stratum if: \begin{enumerate} \item the algebra $E = F[\beta]$ is a field,
\item $E^\times \subset \mathfrak{K}(\mathfrak{A})$, 
\item $v_\mathfrak{A}(\beta) = -n$,
\item $r < k_0(\beta, \mathfrak{A})$.
\end{enumerate}
\end{defi}

\bigskip

\begin{rema} The easiest way to construct a simple stratum is to choose for $\beta$ a minimal element [\cite{BK}, p.44].  In the following, to construct our simple stratum, we shall take $r= 0$ and it will be very easy to check that our chosen $\beta$ is a minimal element. \end{rema}

\bigskip

Now, we can construct from an extension $E$ a simple stratum. Fix $\varpi$ verifying the equation $\varpi^m= a\varpi_F$ with $a \in \OO_F^\times$ that we may choose later in order to control the self-torsion group. Let $E= F[\varpi]$ be a field extension of $F$ of degree $m$. We use Gouvea's book \cite{gouvea} for the construction and properties of finite field extensions of a $p$-adic field, especially Chapter $5$. Denote by $v_E$ the unique valuation on $E$ extended from $v_F$; we have for all $x \in E$, $$v_E(x) =\frac{1}{m}v_p(N_{E/F}(x)),$$ with $N_{E/F} : E \longrightarrow F$ the norm map from $E$ to $F$. We also have $$v_E(\varpi) =\frac{1}{m}.$$ Therefore, $$v_E(\cdot) = \frac{1}{m}\Z.$$ This shows that the ramification index of $E/F$ is $m$ and therefore the inertia degree $f $ is $1$. Hence, by definition, $\varpi$ is a uniformizer in $E$ that we shall denote $\varpi_E$ and $E/F$ is totally ramified. Denote by $\OO_E$ the ring of integers of $E$ and by $\mathfrak{p}_E$ the unique maximal ideal of $\OO_E$ and write $k_E = \OO_E/\mathfrak{p}_E$ for the residue class field of $E$. Since $f= \frac{m}{e}$ is the index $[k_E : k_F]$, it follows that $k_E = k_F$.

Set $\beta = \varpi_E^{-1}$. This is a minimal element by Definition 1.4.14 of \cite{BK}; i.e. it satisfies: $$\gcd(mv_E(\beta), e) = \gcd(-1, m) = 1  \qquad \mathrm{ and} \qquad k_E/k_F \mathrm{\;is\;trivial}.$$ Since $E= F[\varpi_E]$, we also have $E = F[\beta]$. Moreover, since the polynomial $X^m - a\varpi_F$ satisfies Eisenstein's criterion, we have $\OO_E = \OO_F[\varpi_E]$.
 
 \bigskip 
 
 Let $\mathcal{L} = \{ L_i = \varpi_E^i\OO_E, \; i \in \Z\}$ be the unique, up to translation on the index, maximal $\OO_F$-lattice chain. Since $\varpi_E^m = a \varpi_F$ with $a \in \OO_F^\times$, we have for all $i\in \Z$, $$L_{i+m} = \varpi_F L_i.$$ 
 Now, given $\mathfrak{A}$ as in Definition \ref{definition}, $$\mathfrak{A} =  \{ x \in \mathcal{M}_m(F) \; | \; x L_i \subset L_i, i \in \Z\},$$ let us check that $E^\times \subset \mathfrak{K}(\mathfrak{A})$. Let $x \in E^\times$ and $y \in \mathfrak{A}$. Since $E$ is a field, $x$ is invertible and therefore, $x \in G= \mathrm{Aut}_F(G)$. Let $L_i \in \mathcal{L}$. Then \begin{align*} x^{-1}y x L_i & \subseteq x^{-1}y L_{i+k} \qquad \mathrm{for \; a \; certain \;} k \in \Z \\
 & \subseteq x^{-1} L_{i+k} \qquad ( y \in \mathfrak{A}) \\
 & \subseteq L_i. \end{align*} Finally, with respect to the constructed order, we have $\beta L_i = L_{i-1}$ and $\beta \notin \mathfrak{A}$. Therefore, $\nu_\mathfrak{A}(\beta)= -1$. This proves the following proposition: 
 
 \bigskip
 
\begin{prop} The 4-tuple $[\mathfrak{A}, 1 ,0, \beta]$ is a simple stratum. \end{prop}

\begin{rema}\label{cool} Identifying $A$ with $\mathcal{M}_m(F)$, we have $\mathfrak{A} = \mathcal{M}_m(\OO_F)$ and 
\begingroup\setlength{\arraycolsep}{1,2pt}\renewcommand{\arraystretch}{0.9}  $ \beta =~\begin{pmatrix} 0 & 1 & \dots & 0 \\ \vdots & \ddots & \ddots & \vdots \\ 0 &&\ddots & 1\\ \scriptstyle{(a\varpi_F)^{-1}} & 0 & \dots & 0 \end{pmatrix}$. \endgroup \end{rema}

\subsection{Construction of a simple type} We now use the previous step to construct a compact open subgroup $J$ of $\GL(m,F)$ and a character $\lambda$ of $J$. Fix a nontrivial additive character $\psi : F \longmapsto \C^\times$ that is trivial on $\mathfrak{p}_F$ but not on $\OO_F$ and consider the linear character $\psi_\beta$ of $1 + \mathfrak{P}$ defined by $$\psi_\beta(1+x) = \psi(\mathrm{Tr}_{A/F}(\beta x)), \qquad x \in \mathfrak{P},$$ where $\mathrm{Tr}_{A/F}$ is the trace map. Since $\beta \in \mathfrak{P}^{-1}$, $\psi_\beta$ is a character of $U^1(\mathfrak{A})/ U^2(\mathfrak{A})$  where $$ U^0(\mathfrak{A}) = \mathfrak{A}^\times  \qquad \mathrm{and} \qquad U^n(\mathfrak{A}) = 1 + \mathfrak{P}^n, \qquad n \geq 1,$$ is a sequence of compact open subgroups of $G$.\\

For $k \geq 0$, set 
$$J^k = J^k(\beta, \mathfrak{A}) = \left(\mathfrak{B}_\beta + \mathfrak{P} \right) \cap U^k\left(\mathfrak{A}\right)$$ with $\mathfrak{B}_\beta = B_\beta \cap \mathfrak{A}$, where $B_\beta$ is the $A$-centralizer of $\beta$. This is a special case of \cite{BK},~Definition~3.1.8 with $n=1$. For $k=0$, $J^0$ will be denoted by $J = J(\mathfrak{A}, \beta)$.

\bigskip

\begin{prop} We have $J = \OO_E^\times U^1(\mathfrak{A})$. \end{prop}

\bigskip

\begin{proof} Proposition 3.1.15 from \cite{BK} says that $J^0(\beta, \mathfrak{A}) = U^0(\mathfrak{B}_\beta)J^{\left\lfloor \frac{1+r}{2} \right \rfloor} (\beta, \mathfrak{A})$. This is where $\beta$ being taken minimal is useful: by [\cite{BK}, Proposition, 1.4.15] $r =1$. Therefore, $$J(\beta, \mathfrak{A}) = \mathfrak{B}_\beta^\times J^1 (\beta, \mathfrak{A}).$$ The $A$-centralizer $B_\beta$ of $\beta$ is equal to $E$. Indeed, using the embedding in Remark \ref{embedding}, we have $$B_\beta = \{ T \in \mathrm{End}_F(V) \; | \; T(\beta v ) = \beta T(v), v \in V\}.$$ So $B_\beta$ is the set of $E$-linear endomorphisms, $\mathrm{End}_E(V)$. Since $V$ is an $F$-vector space of dimension~$m$, $\mathrm{dim}_E(V) = 1$. Therefore, $\mathrm{End}_E(V) = E$ and $B_\beta = E$. Because of Remarks \ref{embedding} and \ref{cool}, we have $\mathfrak{B}_\beta = B_\beta \cap \mathfrak{A} = E \cap \mathcal{M}_m(\OO_F) = \OO_E$ and $$\mathfrak{B}_\beta^\times = \OO_E^\times.$$

Now, we can compute $ J^1 (\beta, \mathfrak{A})$, by definition, $J^1 =\left(\mathfrak{B}_\beta + \mathfrak{P}\right) \cap \left( 1+ \mathfrak{P}\right)$. Therefore, $$J^1 = 1+ \mathfrak{P} = U^1(\mathfrak{A})$$ and  $J = \mathfrak{B}_\beta^\times J^1 = \OO_E^\times U^1(\mathfrak{A}).$
\end{proof}

Returning to the character $\psi_\beta$ on $U^1(\mathfrak{A})/ U^2(\mathfrak{A})$, we can extend it trivially to $U^1(\mathfrak{A})$. Let $\phi \in \widehat{\OO_E^\times}$ be an arbitrary character trivial on $ 1 + \mathfrak{p}_E$. Since $\OO_E^\times \cap U^1(\mathfrak{A})= 1 + \mathfrak{p}_E$, we can set $\lambda = \phi \otimes \psi_\beta$, which is a character on $J= \OO_E^\times U^1(\mathfrak{A})$. 

\bigskip

\begin{prop}\label{type} The pair $(J, \lambda)$ constructed above is a simple type in the sense of \textup{[}\cite{BK}, Definition 5.5.10 (a)\textup{]}. \end{prop}

\bigskip 

\begin{proof} Recall that $(J, \lambda)$ is a simple type in $G = \GL(m,F)$ if $\lambda$ is a $\beta$-extension of the character $\psi_\beta$ on $J^1$. According to [\cite{BK}, Definition 5.2.1], this means first that $\lambda_{|J^1} = \psi_\beta$, which is the case by construction, and second, that $\lambda$ is intertwined by $E^\times$.  Let $x \in \mathfrak{P}$ and $y  \in E^\times$. We have \begin{align*} \psi_\beta( y^{-1}(1+x)y) &= \psi(\mathrm{Tr}_{A/F}(\beta y^{-1}xy)) \\
&=   \psi(\mathrm{Tr}_{A/F}( y^{-1}\beta xy)) \qquad (\mathrm{since} \; y \in B_\beta) \\
&=  \psi(\mathrm{Tr}_{A/F}(\beta x)) \\
& = \psi_\beta(1+x). \end{align*} and we can extend $\phi$ trivially to $E^\times$, then $\phi(y^{-1}xy) = \phi(x)$. So $\lambda$ is indeed a simple type.
\end{proof}

\begin{rema} The $\sigma$ that appears in the definition of a type in [\cite{BK}, Definition 5.5.10] is here taken to be trivial. \end{rema}

\bigskip

The following statement is now a consequence of [\cite{BK}, Theorem 6.2.2]: 

\medskip

\begin{theo} Any irreducible representation $\pi$ of $G=\GL(m, F)$ containing $\lambda$ is supercuspidal. Moreover, for any such representation $\pi$, there is a uniquely determined representation $\Lambda$ of $E^\times J$ such that $\Lambda_{|J} = \lambda$ and $$\pi = c-\Ind_{E^\times J}^G \Lambda,$$ where $c-\Ind_{E^\times J}^G$ denotes compact induction from $E^\times J$ to $G= \GL(m,F)$. \end{theo}

\bigskip 

Let us note that a representation $\Lambda$ of $E^\times J$ such that $\Lambda_{|J} = \lambda$ is uniquely determined by a choice of $\Lambda (\varpi_E)$ since $E^\times J$ is generated by $\varpi_E$ and $J$.

\bigskip

\subsection{Self-torsion group of $\pi$}

 First, we give a concrete description of the self-torsion group of a supercuspidal representation of the kind that we have just constructed. After this, we will fix two extensions $E_1$ and $E_2$ to obtain $\pi_{m+1}$ and $\pi_{m+2}$ and control their self-torsion groups. Then, we will give a quick discussion on $\pi_1$ and its self-torsion group.

\bigskip 

\begin{prop}\label{propselftorsiongroup} Let $\pi= c-\Ind_{E^\times J}^G \Lambda$ be a supercuspidal representation obtained by the previous construction. Then
$$X(\pi) = \{ \chi \in \widehat{F^\times} \; | \; \chi \circ N_{E/F} =1 \}.$$
\end{prop}

\bigskip

\begin{lemm} Let $\pi= c-\Ind_{E^\times J}^G \Lambda$ be as above. Then $\chi \in X(\pi)$ if and only if~$\Lambda \simeq~\Lambda \otimes~(\chi \circ~\det_A)_{|E^\times J}$.  \end{lemm}

\bigskip

\begin{proof} First, by definition of compact induction, we have $$\left( c-\Ind_{E^\times J}^G \Lambda \right) \otimes( \chi \circ \det) \simeq c-\Ind_{E^\times J}^G \left( \Lambda \otimes( \chi \circ \det)_{|E^\times J} \right).$$ If $\Lambda \simeq \Lambda \otimes (\chi \circ \det)_{|E^\times J}$, then $\pi \simeq \pi \otimes ( \chi\circ \det)$. Conversely, if $\pi \simeq \pi \otimes ( \chi\circ \det)$, then $$\left( c-\Ind_{E^\times J}^G \Lambda \right)  \simeq c-\Ind_{E^\times J}^G \left( \Lambda \otimes( \chi \circ \det)_{|E^\times J} \right).$$ We write $\Lambda \chi$ (resp. $\lambda\chi$) instead of $\Lambda \otimes( \chi \circ \det)_{|E^\times J}$ (Resp.$\lambda \otimes( \chi \circ \det)_{|E^\times J}$). By [\cite{BK}, Theorem 6.2.4], the two simple types $(J, \lambda)$ and $(J, \lambda \otimes \chi)$ contained in $\pi$ are conjuguate in $G$; there exists $x \in G$ such that $^x\lambda = \lambda \chi$, where~$^x\lambda(\cdot) =~\lambda(x^{-1} \cdot x )$, this is an equality since~$\lambda$~is one-dimensional. By [\cite{BKSLII}, Proposition 2.2 and Remark], there exists $x_0 \in U(\mathfrak{A})$, such that $^x\lambda = \lambda \chi$ and we can write $x = x_0y$ with $y \in G$ such that $^y\lambda = \lambda$. Recall that~$\psi_\beta$~is the character we use to construct~$\lambda$~and denote by $I_G(\psi_\beta)$ the subgroup of $G$ that intertwines~$\psi_\beta$. Using Proposition 3.3.1 of \cite{BK}, we have that~$J$~is included in~$I_G(\psi_\beta)$ and we saw in the proof of Proposition \ref{type} that $E^\times \subset I_G(\psi_\beta)$. Therefore,~$E^\times J \subset~I_G(\psi_\beta)$ and $I_G(\psi_\beta)=E^\times J$. Hence, $y \in E^\times J$. Whence $y \in E^\times J$. Still by the remark after [\cite{BKSLII}, Proposition 2.2], $x_0 \in \widetilde{J}$ with $\widetilde{J}$ defined in [\cite{BKSLII}, p. 326]. Since $p \not \! | \;m$, by [\cite{BKSLII}, Corollary 2.6], $\widetilde{J} = J$. Hence $x_0 \in J$. Therefore $x_0 \in E^\times J$ and $x \in E^\times J$. Therefore, $\lambda = \:^x \lambda \simeq \lambda \chi$. Whence, $\Lambda = \Lambda \chi$. This proves the Lemma.
\end{proof}

\bigskip

\begin{proof}[Proof of Proposition \ref{propselftorsiongroup}] Let $\chi \in \widehat{F^\times}$ be such that $\pi \simeq \pi \otimes \chi\circ N_{E/F}$. Then by the previous Lemma, we have $\Lambda \simeq~\Lambda \otimes~(\chi \circ~\det_{|E^\times J})$. Note that, our constructed $\Lambda$ is in fact a character of~$E^\times J$. Therefore, $\pi \simeq \pi \otimes( \chi\circ\det)$ if and only if, $\chi \circ \det_{|E^\times J}=1$. We need to compute the image of $E^\times J$ by the determinant map on $E^\times J \subset A$. First, by definition of the norm map, we have $\det(E^\times) = N_{E/F}(E^\times)$ where the determinant of an element $a$ of $E^\times$ is seen as the determinant of the homomorphism $\mu_a$ given in Remark \ref{embedding}. Secondly, since $E$ is tame, $\det( U^1(\mathfrak{A})) = \det (1+ \mathfrak{P}) \subset 1 + \mathfrak{p}_F$. Hence, $$\det (E^\times J) = \det (E^\times U^1(\mathfrak{A})) \subset N_{E/F}(E^\times) (1 + \mathfrak{p}_F) \subset N_{E/F}(E^\times).$$ Conversely, we have $N_{E/F}(E^\times) \subset \det (E^\times J)$. Therefore, $\pi \in X(\pi)$ if and only if, $\chi \circ N_{E/F}= 1$ as claimed.
\end{proof}

\bigskip

The previous proposition suggests that we need to control the norm map $N_{E/F} : E \longrightarrow F$~and therefore the extension $E$ in order to control $X(\pi)$. \\

We now fix explicit extensions $E_1$ and $E_2$. First, set $E_1 = F[\varpi_{E_1}]$ where $\varpi_{E_1}$ satisfies the equation $X^m + (-1)^{m+1} \varpi_F=0$. Denote by $\pi_{m+1}$ a representation obtained by the previous construction applied to $E_1$ and $\beta_1 = \varpi_E^{-1}$. We shall compute the image of $E_1^\times$ by $N_{E_1/F}$ by using several properties of the field $E_1$ and its norm map $N_{E_1/F}$. Let $\zeta$ be a generator of the group $\mu_{q-1}$ of $(q-1)$-roots of unity. Since $E_1^\times \cong \left< \varpi_{E_1} \right> \times \mu_{q-1} \times \left( 1 + \mathfrak{p}_{E_1} \right)$, we need to compute $N_{E_1/F}(\varpi_{E_1})$, $N_{E_1/F}(\zeta)$ and $N_{E_1/F}(1 + \mathfrak{p}_{E_1})$. We have: \begin{itemize} \item $N_{E_1/F}(\varpi_{E_1}) = \det_A(\varpi_{E_1}) = ((-1)^{m+1})^2 \varpi_F = \varpi_F$, \item since $k_{E_1} = k_F$, it follows that $\mu_{q-1} \subset F^\times$ and hence, $N_{E_1/F}(\zeta)= \zeta^m$, \item since $p \! \not |  \;m$, $E_1/F$ is tamely ramified and therefore, $N_{E_1/F}(1 + \mathfrak{p}_{E_1})= 1 + \mathfrak{p}_F$. \end{itemize} Hence,  $$N_{E_1/F}(E_1^\times) = \left< \varpi_F \right> \times  \mu_{q-1}^m \times (1+ \mathfrak{p}_F).$$

Since $m \; | \; q-1$, set $\eta \in \wF$ such that $$\eta(\varpi_F) =1, \qquad  \eta(1 + \mathfrak{p}_F)=1  \quad\mathrm{and} \quad \eta(\zeta) = e^{2i \pi  / m}.$$ Then $\eta \in X(\pi_{m+1})$ and since $\ker ( \eta) = N_{E_1/F}(E_1^\times)$, we have $X(\pi_{m+1}) = \left< \eta \right>$.



\bigskip

Now, set  $E_2 = F[\varpi_{E_2}]$ such that $\varpi_{E_2}^m + (-1)^{m+1}\zeta \varpi_F=0$ with $\zeta$ as above. We have $$N_{E_2/F}(E_2^\times) = \left< \zeta \varpi_F \right> \times  \mu_{q-1}^m \times (1+ \mathfrak{p}_F).$$ Let $\chi$ be the unramified character of $F^\times$ such that $\chi(\varpi_F)= e^{-2i\pi /m}$. We have $\ker (\chi \eta) = N_{E_2/F}(E_2^\times)$ because $(\chi \eta)(\zeta \varpi_F) = \chi(\varpi_F) \eta(\zeta) = 1$. Then $\chi\eta \in X(\pi_{m+2})$ and $X(\pi_{m+2}) = \left< \chi\eta \right>$. This completes the proof of Proposition \ref{main}, modulo the choice of $\pi_1$.


\bigskip

\subsection{Discussion on $\pi_1$}

There are several possibilities for the choice of $\pi_1$ in Proposition \ref{main}. The first one is to take a generalized Steinberg representation (\cite{steinberg}) $St \otimes \alpha =~\left( \Ind_B^G \;1 / \sum_{P \supsetneq B}^G \;1 \right) \otimes \alpha$ with $\alpha \in \wF$. In this case, $X(\pi_1)$ is trivial, but this is not a supercuspidal representation and then neither is $\pi$, given in \eqref{pii}. Another possibility for $\pi_1$ is using the same construction described before to obtain a supercuspidal representation of $\GL(k)$. For instance, set $E_3 = F[\varpi_{E_3}]$ such that $\varpi_{E_3}$ verifies the equation $X^k + \varpi_F =0$. We have $$N_{E_3/F}(E_3^\times) = \left< \varpi_F \right> \times  \mu_{q-1}^k \times (1+ \mathfrak{p}_F).$$  We just need to check that $\eta$ defined above is not in $X(\pi_1)$ i.e. that $N_{E_3/F}(E_3^\times) \not\subset \ker(\eta)$. If $\gcd (k, q-1) = 1$, then $\mu_{q-1}^k = \mu_{q-1}$ and every character in $X(\pi_1)$ is trivial. Otherwise, if $\gcd (k, q-1 )=d >1$, we have $\mu_{q-1}^k \subset \mu_{q-1}^m$ if and only if $\gcd(q-1,m) \; | \; \gcd(q-1,k)$. In this case, $\left< \eta \right> \subset X(\pi_1)$. Then, at the beginning of section \ref{2}, we shall choose $m$ and $k$ (and eventually $F$) such that $\gcd(q-1,m) \; | \; \gcd(q-1,k)$ if possible; otherwise we can choose a Steinberg representation for $\pi_1$ but $\pi$ in \eqref{pii} will not be supercuspidal.

\section{Prime and small dimensions.}\label{4}
\subsection{Prime dimension}

Let $n$ be an odd prime number. Let $n_1, \dots, n_r \in \N^*$ be a partition of $n$ such that $n= n_1 + \dots + n_r$. Consider the Levi subgroup $M$ of $G$ associated to this partition which is isomorphic to $\GL(n_1, F) \times \dots \times \GL(n_r, F)$. Our first result is given when there is at least one block of size one in the partition. In this case, for any $\sigma \in \dM$, we will not need to lift $\sigma$ from $M$ to $\tM$ using Proposition \ref{pi} and there will always be a fixed point in the orbit of $\sigma$.

\smallskip

\begin{prop}\label{bloc1} Let $n \geq 2$ be any integer. Let $M$ be a Levi subgroup of $G=\SL(n,F)$ such that $M$ is conjugate by an element of $\tG= \GL(n,F)$ to $$\left[\begin{matrix} M' & \\ & F^\times\end{matrix} \right]_{\det=1}$$ with $M'$ a Levi subgroup of $\GL(n-1,F)$, that is to say, $n_r=1$. Let $\sigma \in \dM$. Then there exists $\sigma' \in \OO$ the orbit of $\sigma$ such that $\sigma'$ is fixed by the action of $W_\OO$. 
\end{prop}

\begin{proof} First, the map $$\begin{array}{ccc} M' &\overset{\phi}{\longrightarrow} & M \\ g &\longmapsto& \begin{pmatrix} g & \\ & \det g^{-1} \end{pmatrix} \end{array}$$ gives an isomorphism between $M$ and $M'$. Let $\sigma \in \dM$ and suppose $M'$ is associated to the partition $n-1 = n_1 + \dots + n_{r-1}$, that is to say, 
\begingroup\setlength{\arraycolsep}{1,5pt}\renewcommand{\arraystretch}{1} $$M' \cong \left[ \begin{matrix} \GL(n_1) &&&\\ & \GL(n_2) &&\\&& \ddots& \\ &&& GL(n_{r-1}) \end{matrix} \right].$$ \endgroup Then there exists $\dot\sigma \in \mathcal{E}^2(M')$ such that, $\sigma \simeq \dot\sigma \circ \phi^{-1}$ and $\dot\sigma$ is equivalent to $\sigma_1 \otimes \dots \otimes \sigma_{r-1}$ with $\sigma_i \in \mathcal{E}^2(\GL(n_i))$. Hence $\sigma \simeq \sigma_1 \otimes \dots \otimes \sigma_{r-1} \otimes \sigma_{r}$ with $\sigma_{r}$ the trivial character of $F^\times$.

Now, we use an argument from \cite{AAAMA} to construct a fixed point. Let $i_0 \in  [\![1 , r]\!]$ and let $$\Omega_{i_0} = \left\{ j \in [\![1 , r+1]\!] \; | \; \exists \chi_j \in \mathcal{X}_{nr}^u(F^\times) \st  \sigma_j \simeq \sigma_{i_0} \chi_j\right\}$$ be the orbit of $i_0$ in $W_\OO$. Set $$\chi_{\Omega_{i_0}} = \bigotimes_{j \in \Omega_{i_0}} \chi_j^{-1},$$ and define $\sigma_{|\Omega_{i_0}} \simeq \bigotimes_{j \in \Omega_{i_0}} \sigma_j$ then we have $$\sigma_{|\Omega_{i_0}} \otimes \chi_{\Omega_{i_0}} \simeq \bigotimes_{j \in \Omega_{i_0}} \sigma_{i_0}.$$

For every orbit $\Omega$ of $W_\OO$, we choose a representative $i_\Omega$ of the orbit and define $\chi_\Omega$ the same way. Now define $$  \lambda = \bigotimes_\Omega  \chi_\Omega \qquad \mathrm{and \; set} \qquad \sigma' = \sigma \otimes\lambda.$$ Then $\sigma'$ is a fixed point under the action of $W_\OO$. Indeed, let $w \in W_\OO$. For all orbits $\Omega$, we have $$\sigma'_{|\Omega} = \bigotimes_{j \in \Omega} \sigma_j \otimes \chi_j^{-1} = \bigotimes_{j \in \Omega} \sigma_\Omega$$ with $\sigma_\Omega$ a chosen reprensentative of the orbit $\Omega$. Then, $$^w\sigma'_{|\Omega} =\;  ^w\!\left( \bigotimes_{j \in \Omega} \sigma_\Omega\right) = \bigotimes_{j \in \Omega} \sigma_\Omega = \sigma'_{|\Omega}.$$ Then $\sigma'$ is fixed by all elements of $W_\OO$.
\end{proof}

\bigskip 

The following proposition shows that when $n$ is a prime number, we can simplify the description of the group $W_\OO$ to remove all complications related to the presence of ramified characters.

\bigskip

\begin{prop}\label{prime} Let $n \in \N^*$ be an odd prime number. Let $n = a_1n_1 + \dots + a_rn_r$ where the $a_i$, $n_i$ are integers greater than 1. Let $M$ be the Levi corresponding to the partition with $a_i$ blocks of size $n_i$ for each $i$. Let $\sigma \in \dM$.  By Proposition \ref{pi}, there exists $\pi_\sigma \in \dtM$ such that $\sigma \subset \pi_{\sigma|M}$. Let $w \in W_\OO$ such that there exists $\eta \in \wF$ totally ramified and $\chi \in \XtMnru$ such that $^w\pi_\sigma \simeq \pi_\sigma \chi\eta$. Then $^w\pi_\sigma \simeq \pi_\sigma \chi$; that is, $\eta$ can be taken to be trivial.

\end{prop}

\bigskip

\begin{rema}\label{condblocks} In the previous partition, we can assume that $n_i \geq 2$ for all $1 \leq i \leq r$; otherwise, the question of fixed points is adressed by Proposition \ref{bloc1}. We may also assume that there exists $1 \leq i \leq r$ such that $a_i \geq2$; otherwise $W_\OO = \{1 \}$.  \end{rema}

\bigskip

Before proving Proposition \ref{prime}, we need the following lemma:

\bigskip

\begin{lemm}\label{ordre} Let $m$ be any integer greater than $1$. Let $\pi_1 \in \mathcal{E}^2(\GL(m,F))$. Let $\mu \in \wF$ be a totally ramified character such that there exists $\lambda \in \mathcal{X}_{nr}^u(F^\times)$ such that $\pi_1 \simeq \pi_1 \lambda \mu$. Then, $\mu^m=1$.
\end{lemm}

\bigskip

\begin{proof}
Let $\mu \in \wF$ be as in the lemma. Passing to the the central character, we have $$\lambda^m\mu^m = 1.$$ Then $\mu^m = \lambda^{-m}$ or $\mu^m = \lambda^m =1$. The equality $\mu^m = \lambda^{-m}$ with nontrivial $\mu^m$ and $\lambda^m$ is impossible because $\mu$ is totally ramified. Therefore,  $\mu^m = \lambda^m =1$.
\end{proof}

\bigskip

\begin{rema} If $\mu$ is ramified but not totally ramified, some issues may arise. For instance, take $\lambda = e^{i\frac{\pi}{4}val(\cdot)}$ and $\mu = e^{i\frac{\pi}{4}\val(\cdot)}\eta$ with $\eta$ the Legendre symbol. Then $\left(\lambda \mu\right)^2 = 1$ but $\mu^2 \neq 1$. \end{rema}

\bigskip

\begin{proof}[Proof of Proposition \ref{prime}] Let $n = a_1n_1 + \dots + a_rn_r$ be a partition of $n$. Since $n$ is a prime number, $\gcd(a_1n_1, \dots, a_rn_r) = 1$ This information shall be useful at the end of the proof. Let $\tM$ be Levi subgroup of $\GL(n,F)$ associated to this partition. Let $\sigma \in \dM$. The Weyl group of $\tM$ is isomorphic to $$\mathfrak{S}_1 \times \dots \times \mathfrak{S}_r.$$

Let $w \in W_\OO$, by Proposition \ref{WTheta}, there exists $\eta \in \wF$ totally ramified and $\chi \in \XtMnru$ such that $^w\pi_\sigma \simeq \pi_\sigma \chi\eta$. We can write $w$ as a product of disjoint cycles: $$w = \prod_{i=1}^r \prod_{k=1}^{k_i} c_{i,k}$$ with $c_{i,k}$ the $k$-th cycle that acts on blocks of size $n_i$. Denote by $l_{i,k}$ the length of the cycle $c_{i,k}$, with possibly $l_{i,k} = 1$ for fixed points of $w$. Fix $1 \leq i \leq r$ and let $1 \leq k \leq k_i$. Then $$w_{|c_{i,k}}^{l_{i,k}} =  c_{i,k}^{l_{i,k}} = \mathrm{id}.$$ Since $$^{w^{l_{i,k}}}\pi_\sigma \simeq \pi_\sigma \otimes \left(\bigotimes_{m=0}^{l_{i,k}-1} \; ^{w^m} \!\chi \right) \otimes \eta^{l_{i,k}},$$  we see that for all $j \in \supp(c_{i,k})$, we have \begin{equation}\label{pij}\pi_j \simeq \pi_j \otimes  \left(\bigotimes_{m=0}^{l_{i,k}-1} \; \chi_{w^m(j)} \right) \otimes \eta^{l_{i,k}}.\end{equation} Passing to central characters in \eqref{pij}, we find that $$\left(\left(\bigotimes_{m=0}^{l_{i,k}-1} \; \chi_{w^m(j)} \right) \otimes \eta^{l_{i,k}} \right)^{n_i} = 1.$$ Since $\eta$ is totally ramified by Lemma \ref{ordre}, we have $$\eta^{l_{i,k}n_i} = 1 \qquad \mathrm{for \; all \; }1 \leq k \leq k_i .$$ Since $a_i= \sum_{k=1}^{k_i} l_{i,k}$, we deduce that $$\eta^{a_in_i} = \eta^{l_{i,1}n_i} \dots  \eta^{l_{i,k_i}n_i} =1.$$ The final step uses a Bézout's identity: since $\gcd(a_1n_1, \dots, a_rn_r) = 1$, there exist $u_1, \dots u_{r}\in \Z$ such that $$1 = u_1a_1n_1+ \dots + u_ra_rn_r.$$ Therefore, $$\eta = (\eta^{a_1n_1})^{u_1}\dots ( \eta^{a_rn_r})^{u_r}=1.$$ 
\end{proof}

\begin{theo}  Let $n \in \N^*$ be an odd prime number. Let $n = a_1n_1 + \dots + a_rn_r$ where $n_i$ is the size of a block and $a_i$ the number of blocks of size $n_i$. Denote by $M$ the Levi subgroup associated to this partition and let $\sigma \in \dM$. Then there exists $\sigma' \in \OO$, the orbit of $\sigma$, such that $\sigma'$ is fixed by the action of $W_\OO$. \end{theo}

\bigskip 

\begin{proof} By Proposition \ref{pi}, let $\pi_\sigma \in \dtM$ be such that $\sigma = \pi_{\sigma_{|W}}$ with $W$ a subrepresentation of $\pi_\sigma$. The previous proposition shows that if $n$ is a prime number, then we have the following simpler description of $W_\OO$:  $$W_\OO = \{w \in W(M), \exists \chi \in \XtMnru \st ^w\pi_\sigma \simeq \pi_\sigma \otimes \chi \}.$$ The rest of the proof is now very similar to the argument used in Proposition \ref{bloc1}. Let~$i_0\in~\{1, \dots, n \}$ and define $\Omega_{i_0}$, the orbit of $i_0$ under $W_\OO$, $$\Omega_{i_0} = \left\{ j \in [\![1 , n]\!] \; | \; \exists \chi_j \in \mathcal{X}_{nr}^u(F^\times) \st  \pi_j \simeq \pi_{i_0} \chi_j\right\}.$$ Set $$\chi_{\Omega_{i_0}} = \bigotimes_{j \in \Omega_{i_0}} \chi_j^{-1}, \qquad \lambda = \bigotimes_\Omega  \chi_\Omega \qquad \mathrm{and} \qquad \pi'_{\sigma} =  \pi_{\sigma} \otimes \lambda.$$ Then $ \pi'_{\sigma}$ is fixed under the action of $W_\OO$. Indeed, let $w \in W_\OO$. For all orbits $\Omega$, we have $$\left(\pi'_{\sigma}\right)_{|\Omega} = \bigotimes_{j \in \Omega} \pi_j \otimes \chi_j^{-1} = \bigotimes_{j \in \Omega} \pi_\Omega$$ with $\pi_\Omega$ a chosen reprensentative of the orbit $\Omega$. Hence $$\left(^w\pi'_{\sigma}\right)_{|\Omega} =\;  ^w\!\left( \bigotimes_{j \in \Omega} \pi_\Omega\right) = \bigotimes_{j \in \Omega} \pi_\Omega = \left(\pi'_{\sigma}\right)_{|\Omega}.$$ Therefore $\pi'_{\sigma}$ is fixed $w$. Now, define $\sigma' = \pi_{\sigma{|W}} \otimes \lambda_{|M}$. Then $\sigma' \in \OO$, since $\sigma' = \sigma \otimes \lambda_{|M}$. Moreover, $\sigma'$ is a subrepresentation of $\pi'_{\sigma|M} = \pi_{\sigma|M} \otimes \lambda$. Let $w \in W_\OO$. Since $^w\pi'_{\sigma} \simeq \pi'_{\sigma}$, we have $$^w\sigma' \simeq \; ^w\pi'_{\sigma|W}\simeq \pi'_{\sigma|W} \simeq \sigma'.$$ Therefore $\sigma'$ is a fixed point under the action of $W_\OO$.
\end{proof}

\begin{rema}\label{coro} If $n$ is not assumed to be prime but if $M$ is any Levi of $G= \SL(n,F)$ and $\sigma \in \dM$ such that $W_\OO = \{ w \in W(M), \exists \chi \in \XtMnru \st ^w\pi_\sigma \simeq \pi_\sigma \otimes \chi \}$, then the proof of Proposition \ref{prime} goes through and shows that there exists $\sigma' \in \OO$ such that $\sigma'$ is fixed by the action of $W_\OO$. 
\end{rema}

\subsection{Small dimensions}

In this section we prove, with a case-by-case analysis, that every orbit $\OO$ of $W_\OO$ has a fixed point when $G= \SL(n,F)$ with $n\leq 9$ and $n \neq 8$.

\subsubsection{For $n=2$}

Let $G=\SL(2, F)$. Then $G$ has only one proper Levi subgroup: $$\left\{\begin{pmatrix} x & 0 \\ 0 & x^{-1} \end{pmatrix}, \; x \in F^\times \right\}.$$ Its Weyl group $W(M)$ is isomorphic to $\Z/2\Z$. We denote by $w$ the non trivial element of $W(M)$. Let $\sigma \in \dM$. Since $M$ is isomorphic to $F^\times$, $\sigma$ can be seen as a character of $F^\times$ denoted by $\xi$. Therefore, we have $\sigma\left(  \begin{pmatrix} x & 0 \\ 0 & x^{-1} \end{pmatrix} \right) = \xi(x)$ and $W(M)$ acts on $\xi$ by $^w\xi = \xi^{-1}$.

Assume that $W_\OO = \Z/2\Z$. Then there exists $\chi \in \mathcal{X}_{nr}^u$ such that $^w\xi = \xi \chi$. Define $\xi' = \xi \chi^{1/2}$. Then $$^w\xi' = \left(\xi \chi^{1/2}\right)^{-1}= \xi \chi \chi^{-1/2} = \xi\chi^{1/2} = \xi'.$$ Hence $\xi'$ is a fixed point. 

\subsubsection{For $n=3$}

Let $G= \SL(3,F)$. There are only two partitions giving rise to Levi subgroups of $G$ namely, $n=2+1$ and $n=1+1+1$. In the first case, $W_\OO$ is trivial. And the second one is treated by Proposition \ref{bloc1}. Therefore there is always a fixed point in this case. 

\bigskip 

In the next subsections, concerning $n=4, 6,7$ and $9$, we will only look at partitions that do not have been already covered by Propositions \ref{bloc1} and \ref{prime}, namely: like in remark \ref{condblocks}, partitions that do not have a block of size $1$ and that have at least two blocks of the same size. For any $w \in W_\OO$, we also assume that the chosen $\eta$ associated to $w$ is totally ramified; otherwise, it is treated by Remark \ref{coro}.

\bigskip

\subsubsection{For $n=4$}\label{n4} The only partition for which our previous results do not immediately imply the existence of fixed points is $n=2+2$. Consider the Levi subgroup $$M= \left[\begin{matrix} \GL(2) & 0 \\ 0 & \GL(2) \end{matrix}\right]_{\det=1}.$$ Its Weyl group $W(M)$ is isomorphic to $\Z/2\Z$. Denote by $w$ the non trivial element of $W(M)$. Let $\sigma \in \dM$ and $\pi_\sigma \in \dtM$ be such that $\sigma \subset \pi_{\sigma|M}$, and assume $W_\OO =\Z/2\Z$. We have $\pi_\sigma \simeq \pi_1 \otimes \pi_2$ with $\pi_1,\pi_2 \in \mathcal{E}^2(\GL(2,F))$ and there exists $\eta \in \wF$ totally ramified and $\chi = \chi_1 \otimes \chi_2 \in \XtMnru$ such that $$\left\{\begin{matrix} \pi_2 \simeq \pi_1 \chi_1 \eta \\ \pi_1 \simeq \pi_2 \chi_2 \eta \end{matrix}\right. .$$ Calculating $^{w^2}\!\pi_\sigma$, we have $$\pi_1 \simeq \pi_1 \chi_1 \chi_2 \eta^2.$$

\bigskip

Since $\pi_1 \simeq \pi_1 \chi_1 \chi_2 \eta^2$, by Lemma \ref{ordre}, we have $\eta^4=1$. For the construction of a fixed point, there are two possible cases.

The first case is when $\eta^2 \in X(\pi_1)$. In this case, define $$\pi_\sigma' = \pi_1 \otimes \left(\pi_2 \chi_1^{-1}\right) \simeq \pi_1 \otimes \pi_1\eta$$ then we have $$^w \pi_\sigma' \simeq \left(\begin{matrix} \pi_1\eta & \\ & \pi_1 \end{matrix}\right) \simeq  \left(\begin{matrix} \pi_1 \eta & \\ & \pi_1 \eta^2\end{matrix}\right) \simeq \pi_\sigma' \eta.$$ Hence $\pi_\sigma'$ is fixed under the action of $W_\OO$.

The second case is when $\eta^2 \notin X(\pi_1)$. Define $\pi_\sigma'' \simeq \pi_1\chi_2^{-1/2} \otimes \pi_2\chi_1^{-1/2}$. Then we have \begin{equation*} ^w \pi_\sigma'' \simeq \left(\begin{matrix} \pi_2\chi_1^{-1/2} & \\ &  \pi_1\chi_2^{-1/2} \end{matrix}\right) \simeq  \left(\begin{matrix} \pi_1  \chi_1^{1/2}\eta & \\ & \pi_2\chi_2^{1/2} \eta\end{matrix}\right) \simeq  \left(\begin{matrix} \pi_1\chi_2^{-1/2} & \\ &  \pi_2\chi_1^{-1/2} \end{matrix}\right) \otimes \chi_1^{1/2}\chi_2^{1/2}\eta\simeq  \pi_\sigma'' \otimes \chi_1^{1/2}\chi_2^{1/2}\eta.  \end{equation*} Here, the characters appearing inside the matrices are characters of $\GL(2,F)$ whereas those outside are characters of $F^\times$. Thus $\pi''_\sigma \simeq \pi''_\sigma \otimes \eta'$ for some $\eta' \in \wF$, which proves that $\pi''_\sigma$ is a fixed point under the action of $W_\OO$.

\bigskip

\subsubsection{For $n=6$}\label{n6} There are still two partitions to consider: $n=3+3$ and $n=2+2+2$. For the first partition, let $M$ be the Levi subgroup associated to $n= 3+3$. Then the same construction given in \ref{n4} can be applied and we obtain the same fixed points depending on if $\eta^2$ is in $X(\pi_1)$ or not. For the second partition, let $M$ be the Levi subgroup associated to $n=2+2+2$. The Weyl group of $M$ is isomorphic to $\mathfrak{S}_3$. Let $\sigma \in \dM$ and $\pi_\sigma \in \dtM$ be such that $\sigma \subset \pi_{\sigma|M}$ with $\pi_\sigma \simeq \pi_1 \otimes \pi_2 \otimes \pi_3$. Then there are several possibilities for $W_\OO$: 
\begin{itemize}
\item If $W_\OO \cong \Z/2\Z$, denote by $w$ its non trivial element. Then there exists $\eta \in \wF$ totally ramified and $\chi = \chi_1 \otimes \chi_2 \otimes \chi_3 \in \XtMnru$ such that $$\left\{\begin{matrix} \pi_2 \simeq \pi_1 \chi_1 \eta \\ \pi_1 \simeq \pi_2 \chi_2 \eta \\ \pi_3 \simeq \pi_3 \chi_3\eta\end{matrix}\right.$$

By Lemma \ref{ordre} we know, using the third equality, that $\eta^2=1$. Set $\pi_\sigma' = \pi_1 \otimes \pi_1 \chi_3 \eta \otimes \pi_3$. This is a fixed point under the action of $W_\OO$ because since $(\chi_3\eta)^2= 1$, we have $$^w \! \pi_\sigma' \simeq \left( \begin{matrix} \pi_1\chi_3 \eta &&\\ & \pi_1 \\&& \pi_3 \end{matrix} \right) \simeq \left( \begin{matrix} \pi_1\chi_3 \eta &&\\ & \pi_1\chi_3^2\eta^2 \\&& \pi_3  \chi_3 \eta \end{matrix} \right) \simeq \pi_\sigma' \otimes \chi_3\eta,$$ and $\pi'_\sigma$ is a fixed point as before.

\item If $W_\OO \cong \Z/3\Z$, denote by $w$ a generator of $W_\OO$. Then there exists $\eta \in \wF$ totally ramified and $\chi = \chi_1 \otimes \chi_2 \otimes \chi_3 \in \XtMnru$ such that $$\left\{\begin{matrix} \pi_2 \simeq \pi_1 \chi_1 \eta \\ \pi_3 \simeq \pi_2 \chi_2 \eta \\ \pi_1 \simeq \pi_3 \chi_3\eta\end{matrix}\right.$$

As in \ref{n4} there are two cases. If $\eta^3 \in X(\pi_1)$ set $\pi_\sigma' \simeq \pi_1 \otimes \pi_1\eta \otimes \pi_1\eta^2$. Then this is a fixed point. If $\eta^3 \notin X(\pi_1)$, set $\pi_\sigma'' \simeq \pi_1 \chi_1^{1/3}\chi_3^{-1/3} \otimes  \pi_2 \chi_2^{1/3}\chi_1^{-1/3} \otimes  \pi_3 \chi_3^{1/3}\chi_2^{-1/3}$ then  
\begingroup\setlength{\arraycolsep}{0,5pt}\renewcommand{\arraystretch}{0.7}  \begin{align*}^w\!\pi_\sigma'' & \simeq \left( \begin{matrix} \pi_2 \chi_2^{1/3}\chi_1^{-1/3}&& \\&  \pi_3 \chi_3^{1/3}\chi_2^{-1/3} & \\ &&  \pi_1 \chi_1^{1/3}\chi_3^{-1/3} \end{matrix} \right) \\ 
& \simeq \left( \begin{matrix} \pi_1 \chi_1^{2/3} \chi_2^{1/3}\eta && \\&  \pi_2 \chi_2^{2/3} \chi_3^{1/3} \eta & \\ &&  \pi_3 \chi_3^{2/3} \chi_1^{1/3}\eta \end{matrix} \right)  \\
&\simeq  \left( \begin{matrix}  \pi_1 \chi_1^{1/3}\chi_3^{-1/3} && \\&\pi_2 \chi_2^{1/3}\chi_1^{-1/3}& \\ &&\pi_3 \chi_3^{1/3}\chi_2^{-1/3} \end{matrix} \right) \otimes \chi_1^{1/3}\chi_2^{1/3} \chi_3^{1/3} \eta \\ & \simeq \;  \pi''_\sigma \otimes \chi_1^{1/3}\chi_2^{1/3} \chi_3^{1/3} \eta \end{align*} \endgroup and $\pi_\sigma''$ is a fixed point as before.

\item If $W_\OO \cong \mathfrak{S}_3$, then there exists $\eta, \mu \in \wF$ (totally ramified) and $\chi = \chi_1 \otimes \chi_2 \otimes \chi_3$, $\lambda = \lambda_1 \otimes \lambda_2 \otimes \lambda_3\in \XtMnru$ such that $^{(1\; 2 \; 3)} \pi_\sigma \simeq \pi_\sigma \chi \eta$ and $^{(1\; 2)} \pi_\sigma \simeq \pi_\sigma \lambda \mu$. Therefore $$\left\{\begin{matrix} \pi_2 \simeq \pi_1 \chi_1 \eta \\ \pi_3 \simeq \pi_2 \chi_2 \eta \\ \pi_1 \simeq \pi_3 \chi_3\eta\end{matrix}\right. \qquad \mathrm{and} \qquad \left\{\begin{matrix} \pi_2 \simeq \pi_1 \lambda_1 \mu \\ \pi_1 \simeq \pi_2 \lambda_2 \mu \\ \pi_3 \simeq \pi_3 \lambda_3\mu\end{matrix}\right. .$$
\end{itemize}

By computing the actions of $(1\; 2 \; 3)(1\; 2)$ and $(1\; 2)(1\; 2 \; 3)$ on $\pi_\sigma$, we find $$\left\{\begin{matrix} \pi_3 \simeq \pi_1 \lambda_2 \chi_1 \eta \mu \\ \pi_2 \simeq \pi_2 \lambda_3 \chi_2 \eta \mu \\ \pi_1 \simeq \pi_3 \lambda_1 \chi_3\eta \mu \end{matrix}\right. \qquad \mathrm{and} \qquad \left\{\begin{matrix} \pi_1 \simeq \pi_1 \lambda_1 \chi_2 \eta \mu \\ \pi_3 \simeq \pi_2 \lambda_2 \chi_1 \eta \mu \\ \pi_2 \simeq \pi_3 \lambda_3 \chi_3 \eta \mu\end{matrix}\right. .$$ Then combining the first equality on the left hand side with the second one on the right hand side, and the third equality on the left with the first on the right, we obtain $$\pi_3 \simeq  \pi_1 \chi_2 \chi_3^{-1} \qquad \mathrm{and} \qquad \pi_2 \simeq \pi_1.$$ Then we can choose $\eta =\mu =1$ and use Proposition \ref{prime} and Remark \ref{coro} to find as a fixed point $$\pi_\sigma' = \pi_1 \otimes \pi_1 \otimes \pi_1.$$

\bigskip

\subsubsection{For $n=9$} There are two partitions to look at: $n=2+2+2+3$ and $n=3+3+3$. Let 
\begingroup\setlength{\arraycolsep}{1,5pt}\renewcommand{\arraystretch}{0.9}  $$M =  \left[ \begin{matrix} \GL(2) &&& \\ & \GL(2) && \\ && \GL(2) & \\ &&& \GL(3) \end{matrix} \right]_{\det=1}$$ \endgroup be a Levi subgroup associated to the first partition. Its Weyl group $W(M)$ is isomorphic to a subgoup of $\mathfrak{S}_3$.  Let $\sigma \in \dM$ and $\pi_\sigma \in \dtM$ be such that $\sigma \subset \pi_{\sigma|M}$ with $\pi_\sigma \simeq \pi_1 \otimes~\pi_2 \otimes~\pi_3 \otimes~\pi_4$. Then there are several possibilities for $W_\OO$: 
\begin{enumerate}[(a)]
\item If $W_\OO = \{ 1 , (1 \; 2)\}$, there exists $\eta \in \wF$ and $\chi \in \XtMnru$ such that $^{(1 \; 2)} \pi_\sigma \simeq \pi_\sigma \chi \eta$. We use Lemma \ref{ordre} on $\pi_3$ and $\pi_4$ to prove that $\eta^2= 1$ and $\eta^3=1$. Therefore $\eta = 1$ and we use Remark~\ref{coro}~to find as a fixed point $$\pi_\sigma' = \pi_1 \otimes \pi_1 \otimes \pi_3 \otimes \pi_4.$$

\item If $W_\OO \cong \Z/3\Z$, there exists $\eta \in \wF$ and $\chi \in \XtMnru$ such that $^{(1 \; 2\; 3)} \pi_\sigma \simeq \pi_\sigma \chi \eta$. Using Lemma \ref{ordre} on $\pi_4$, we have $\chi_4^3 \eta^3= 1$ and $\eta^3=1$. Set $\pi_\sigma' = \pi_1 \otimes \pi_1 \chi_4\eta \otimes  \pi_1 \chi_4^2\eta^2 \otimes \pi_4$. Then 
\begingroup\setlength{\arraycolsep}{1,5pt}\renewcommand{\arraystretch}{0.8} $$^{(1 \; 2\; 3)} \pi_\sigma' \simeq \left( \begin{matrix} \pi_1 \chi_4 \eta &&& \\ & \pi_1 \chi_4^2 \eta^2 && \\ && \pi_1 & \\ &&& \pi_4 \end{matrix}\right) \simeq  \left( \begin{matrix} \pi_1 &&& \\ &\pi_1 \chi_4 \eta && \\ && \pi_1 \chi_4^2 \eta^2 & \\ &&& \pi_4\end{matrix}\right)\otimes \chi_4 \eta.$$ \endgroup So $\pi_\sigma'$ is a fixed point.

\item If $W_\OO \cong \mathfrak{S}_3$, $W_\OO$ is generated by $(1\; 2)$ and $(2 \; 3)$ and we saw in case (a) above that any $\eta$~associated to a transposition can be taken equal to 1. Then Proposition \ref{prime} and Remark~\ref{coro}~yields a fixed point that shall be $$\pi_\sigma' = \pi_1 \otimes \pi_1 \otimes \pi_1 \otimes \pi_4.$$
\end{enumerate}

Now, for the second partition; $n=3+3+3$.
\begin{enumerate}[(a)]
\item If $W_\OO \cong \Z/2\Z$. We use Lemma \ref{ordre} on $\pi_3$ to obtain $\eta^3 =1$. Computing $^{(1 \; 2)^3} \pi_\sigma$ we can show that $\eta$ can be taken equal to 1. Indeed, $$^{(1 \; 2)} \pi_\sigma \simeq^{(1 \; 2)^3} \pi_\sigma \simeq \pi_\sigma  {^{(1 \; 2)}}\chi \chi^2\eta^3 \simeq \pi_\sigma {^{(1 \; 2)}}\chi  \chi^2.$$ Therefore $\pi_2 \simeq \pi_1\chi_1^2 \chi_2$ so $\eta$ can be taken to be trivial and we use Remark \ref{coro} to find a fixed point that shall be $$\pi'_\sigma = \pi_1 \otimes \pi_1 \otimes \pi_3.$$
\item If $W_\OO \cong \Z/3\Z$, the exact same construction used in \ref{n6} for $6=2+2+2$ works here.
\item If $W_\OO \cong \mathfrak{S}_3$, $W_\OO$ is generated by $(1\; 2)$ and $(2 \; 3)$ and we saw, when $W_\OO \cong \Z/2\Z$, that the~$\eta$~associated to a transposition can be taken equal to 1. Then Proposition \ref{prime} and Remark \ref{coro} gives us a fixed point, namely $$\pi_\sigma' = \pi_1 \otimes \pi_1 \otimes \pi_3.$$
\end{enumerate}

\appendix
\section{Discussion on characters of order $m$.}\label{appendix}

Let $F$ be a non-archimedean local field. Denote by $ \OO_F$ its ring of integers and by $\mathfrak{p}_F$ the unique maximal ideal of $\OO_F$. Let $q = p^f$ be the cardinal of the residue class field $k_F = \OO_F/ \mathfrak{p}_F$ and $d = [F : \Q_p]$ be the degree of the extension. Denote by $\varpi_F \in \wF$ a uniformizer, we have $F^\times \cong \; \left<\varpi_F \right> \times \; \OO_F^\times$. We want to determine conditions on $m$, $p$ and $q$ such that there exists a character of $F^\times$ ramified or totally ramified of order $m$. First, we need a result on the structure of $F^\times$. 

\bigskip

\begin{prop}[\cite{neukirch}, Proposition 5.7] Let $F$ be a non-archimedean local field. Then with the above notations, we have \begin{enumerate} \item if $F$ has characteristic zero: \begin{align*} F^\times &  \cong \;  \left<\varpi_F \right> \times \mu_{q-1} \times (1 + \mathfrak{p}_F) \\& \cong \Z \oplus \Z/(q-1)\Z \oplus  \Z/p^a\Z \oplus \Z_p^d \end{align*} where $\mu_{q-1}$ is the group of $(q-1)$-th roots of unity, $a \geq 0$ and $\Z_p$ is the ring of integers of $\Q_p$.

\item If $F$ has characteristic $p$:  \begin{align*} F^\times &  \cong \;  \left<\varpi_F \right> \times \mu_{q-1} \times (1 + \mathfrak{p}_F) \\& \cong \Z \oplus \Z/(q-1)\Z  \oplus \Z_p^\N. \end{align*}
\end{enumerate}
\end{prop}

\bigskip

A direct consequence of this proposition is that the group of characters of $F^\times$ is isomorphic to \begin{equation*}\label{dual}\widehat{\Z} \oplus \widehat{\Z/(q\!-\!1)\Z} \oplus \widehat{ \Z/p^a\Z} \oplus \widehat{\Z_p^d},\end{equation*} if $F$ has characteristic zero or, if $F$ has positive characteristic, the group $\wF$ is isomorphic to \begin{equation*}\label{dual}\widehat{\Z} \oplus \widehat{\Z/(q\!-\!1)\Z}  \oplus \widehat{\Z_p^\N}.\end{equation*} Let us to describe each component of $\wF$, depending on the characteristic of $F$. The first component for both characteristics, $\widehat{\Z}$, is isomorphic to $\C^\times$: this component corresponds to the set of unramified characters, namely those that are trivial on $\OO_F^\times$. Next, for the components that are abelian cyclic groups, we have that their group of characters are isomorphic to themselves. For the last component, we have $\widehat{\Z_p^d} \cong \left(\widehat{\Z_p}\right)^d$ and  $\widehat{\Z_p^\N} \cong \left(\widehat{\Z_p}\right)^\N$, this leads us to the following lemma on the dual of $\Z_p$.

\bigskip

\begin{lemm}[\cite{BTS1}, Corollary of Proposition 20 p.237]\label{A2} The group of additive characters of $\Z_p$ is isomorphic to $\Q_p/\Z_p$.\end{lemm}

\bigskip

Now, we discuss the existence of characters of order $m$ for each component. For the first component of $\wF$, the unramified characters, there is always a character of order $m$. Those characters will be associated to a $m$-th root of unity, for instance, the character of  $F^\times$ given by~$x \longmapsto e^{\frac{2i \pi}{m}\val(x)}$ is of order $m$. For the ramified characters of $F^\times$, we only used nontrivial characters of $\mu_{q-1}$, but for the sake of completeness, we still prove the following proposition and remark.

\bigskip

\begin{prop} Assume $F$ is of characteristic zero. Let $m \geq 2$ and note $m = p^n m'$ such that $p \not \! | \;m'$. There exists a ramified character of order $m$ of $F^\times$ if and only if the three following conditions are satisfied: \begin{enumerate} \item there exists $d_{m'}$ a divisor of $m'$ that can be equal to $1$ such that $d_{m'} \; | \; q-1$,
\item there exists $0 \leq n' \leq n$ such that $n' <a$, 
\item there exists $x \in \Q_p/\Z_p$ such that $\mathrm{LCM}(d_{m'}p^{n'}, -v_p(x)) = m$.
\end{enumerate}
\end{prop}

\bigskip

\begin{proof} Let $\chi$ be a character of $F^\times$ of order $m$. By \ref{dual}, we can write $\chi$ as a product of characters denoted by $\chi_{q-1}$, $\chi_{p^a}$ and $\chi_p$ of each components of the ramified part $\wF$. Let $d$ be the order of $\chi_{q-1}$. We have $d \; | \; q-1$ and $d \; | \; m$. Since $p$ does not divides $q-1$, therefore, $d \; | \; m'$ and we note $d = d_{m'}$. Next, $\chi_{p^a}$ has order $p^{n'}$ with $n' \leq n$ and $n'< a$ because $p^{n'}$ divides $m$ and $p^a$. Since $p$ and $q-1$ are coprime, the character $\chi_{q-1}\chi_{p^a}$ has order $d_{m'}p^{n'}$. Denote by $x$ the element of $\Q_p/\Z_p$ associated to $\chi_p$. The order of $\chi_p$ is exactly $-v_p(x)$, then, since $\chi$ has order~$m$, we have $\mathrm{LCD}(d_{m'}p^{n'}, -v_p(x))=m$, which proves the direct implication. The reverse implication is immediate.
\end{proof}

\bigskip

\begin{rema} If $F$ has positive characteristic, the same proposition holds by removing the second condition and $p^{n'}$ in the third condition. \end{rema}

\backmatter

\bibliographystyle{smfalpha}
\bibliography{biblioptfixe}

\end{document}